\documentclass[11pt, a4paper]{amsart}

\usepackage{amsfonts,amsmath,amssymb, amscd}
\usepackage{txfonts, mathabx}
\usepackage[T1]{fontenc}

\makeindex

\newtheorem{theorem}{Theorem}[section]
\newtheorem{lemma}[theorem]{Lemma}
\newtheorem{definition}[theorem]{Definition}
\newtheorem{proposition}[theorem]{Proposition}
\newtheorem{corollary}[theorem]{Corollary}

\theoremstyle{definition}
\newtheorem{remark}[theorem]{Remark}

\newtheorem{example}[theorem]{Example}

\newtheorem{exercise}[theorem]{Exercise}

\newcommand\eps{\varepsilon}

\renewcommand\AA{\mathcal A}
\newcommand\BB{\mathcal B}
\newcommand\OO{\mathcal O}
\renewcommand\SS{\mathcal S}
\newcommand\UU{\mathcal U}

\newcommand\card{\mathrm{card}}

\newcommand\NN{\mathbb{N}}
\newcommand\ZZZ{\mathbb {Z}}

\newcommand\RR{\mathbb{R}}
\newcommand\CC{\mathbb{C}}
\newcommand\HH{\mathbb{H}}
\newcommand\PP{\mathbb{P}}

\newcommand\Hom{\mathrm{Hom}}
\newcommand\End{\mathrm{End}}
\newcommand\Aut{\mathrm{Aut}}
\newcommand\SL{\mathrm{SL}}
\newcommand\ab{\mathrm{ab}}
\newcommand\co{\mathrm{co}}
\newcommand\op{\mathrm{op}}

\newcommand\inv{\mathrm{inv}}

\newcommand\Iso{\mathrm{Iso}}

\newcommand\Frac{\operatorname{Frac}}
\newcommand\Alg{\operatorname{Alg}}
\newcommand\Gr{\operatorname{Gr}}
\newcommand\Gal{\operatorname{Gal}}
\newcommand\ZGal{\operatorname{Zgal}}
\newcommand\mm{\mathfrak{m}}
\newcommand\gog{\mathfrak{g}}
\newcommand\gs{\mathfrak{s}}
\newcommand\gl{\mathfrak{l}}
\newcommand\gu{\mathfrak{u}}

\DeclareMathOperator{\id}{id}

\DeclareMathOperator{\Sym}{Sym}

\numberwithin{equation}{section}

\title[Non-commutative principal fiber bundles]
{Principal fiber bundles in\\ non-commutative geometry}

\author{Christian Kassel}
\address{Christian Kassel, 
Institut de Recherche Math\'e\-ma\-tique Avanc\'ee,
CNRS \& Universit\'e de Strasbourg,
7 rue Ren\'{e} Descartes, 67084 Strasbourg, France}
\email{kassel@math.unistra.fr}



\begin{document}

\begin{abstract}
These are the expanded notes of a course given at the Summer school 
``Geo\-metric, topological and algebraic methods for quantum field theory'' held at Villa de Leyva, Colombia in July 2015. 
We first give an introduction to non-commutative geometry and to the language of Hopf algebras.
We next build up a theory of non-commutative principal fiber bundles and 
consider various aspects of such objects.
Finally, we illustrate the theory using the quantum enveloping algebra~$U_q\, \gs\gl(2)$ and related Hopf algebras.
\end{abstract}

\maketitle

\tableofcontents


\begin{flushright}
\emph{Al \'algebra le dediqu\'e mis mejores \'animos,\\ 
no s\'olo por respeto a su estirpe cl\'asica\\
sino por mi cari\~no y mi terror al maestro.}\\
\smallskip
\textsc{Gabriel Garc\'\i a M\'arquez},\\
{Vivir para contarla}\,\cite{GM}
\end{flushright}

\section{Introduction}\label{sec-intro}

These are the expanded notes of a course given at the Summer school 
``Geometric, topological and algebraic methods for quantum field theory'' held at Villa de Leyva, Colombia in July 2015. 
The main objective of this course was twofold: 
\begin{enumerate}
\item
to give an introduction to non-commutative geometry and to the language of Hopf algebras;

\item
to build up a theory of non-commutative principal fiber bundles, consider various aspects of these non-commutative
objects, highlight the similarities and the differences with their classical counterparts,
and illustrate the theory with significant examples.
\end{enumerate}

Non-commutative geometry is based on the idea that instead of working with the points of a topological space~$X$
(or a $C^{\infty}$-manifold, or an algebraic variety) we may just as well work with the algebra~$\OO(X)$ of 
continuous (or $C^{\infty}$, or regular) functions on~$X$. 
Many geometrical constructions on~$X$ can be expressed by algebraic constructions on the commutative algebra~$\OO(X)$,
which in turn can be extended to non-necessarily commutative algebras. 
The necessity of passing from commutative algebra to non-commutative ones originates from physics; 
according to~\cite{BM1},
\begin{quote}
[it] arises from the general indication that the small-scale structure of space-time is not well-modelled by usual continuous geometry. 
At the Planck scale one may reasonably expect that our notion of geometry has to be modified 
to include quantum effects as well. Non-commutative geometry has the potential to do this.
\end{quote} 
Keeping in mind the geometric origin of such non-commutative constructions, it is natural to use the phrase 
``non-commutative spaces'' for non-commutative algebras. 
In mathematics such generalized spaces have appeared in the 1980's in the work of Connes on group actions 
and on foliations (see~\cite{Co}), but also in the theory of quantum groups, which originated in the work of Faddeev's school, 
of Drinfeld, of Jimbo, and of Woronowicz (see~\cite{Dr1, Dr2, Ji, RTF, Wo}).

Quantum groups are non-commutative algebras depending on a parameter~$q$. 
When $q$ takes the value~$1$, then quantum groups specialize to classical objects such as groups of symmetries.
The construction of quantum groups was inspired by the 
``quantum inverse scattering method'', a method devised for constructing integrable quantum systems
and mostly developped by L. D. Faddeev and his collaborators.
The discovery of quantum groups was a major event with spectacular applications not only in quantum physics, 
but also in domains of pure mathematics such as representation theory and low-dimensional topology.
Let us quote Drinfeld on quantization from the introduction of\,\cite{Dr2}:
\begin{quote}
... both in classical and quantum mechanics there are two basic concepts: state and observable.
In classical mechanics [...] observables are functions on [a manifold]~$M$. In the quantum case [...]
observables are operators in [a Hilbert space]~$H$ [...] 
[O]bservables form an associative algebra which is commutative in the classical case and noncommutative in
the quantum case. So quantization is something like replacing commutative algebras by noncommutative ones.
\end{quote}

Technically speaking, quantum groups are what algebraists and topologists call Hopf algebras. 
Therefore, the first aim of this course was 
to introduce the concept of a Hopf algebra and to illustrate it with significant examples, 
such as the ones related to the special linear group~$SL_2(\CC)$.

Our second aim was to define non-commutative analogues of principal fiber bundles.
Principal fiber bundles are ubiquitous geometrical objects in mathematics and gauge theory.
For instance, given a Lie (or algebraic) group~$G$ and a closed subgroup~$G'$, the projection
$G \to G/G'$ onto the homogeneous space~$G/G'$ is a principal fiber bundle.
To quantize homogeneous spaces we need an adequate notion of quotient of Hopf algebras
and more precisely the concepts of comodule algebras and Hopf Galois extensions.  
There are numerous meaningful examples of non-commutative principal fiber bundles; 
see \cite{BM1, DGH, Ha1, HM, LPR, LS, Pf, Po}.

Let us give an overview of these notes. 
In~Sect.\,\ref{sec-bundle} we review the definition of classical principal fiber bundles and state their main properties. 
In~Sect.\,\ref{sec-NCG} we undertake the crucial passage from commutative to non-commutative algebras;
we concentrate on two simple situations in which a space $X$ can easily be replaced by its function algebra~$\OO(X)$, 
namely when $X$ is a finite set or when it is an affine algebraic variety. 
To make things even simpler, all objects and algebras considered in these notes 
are defined over the field~$\CC$ of complex numbers. 
We also give in~Sect.\,\ref{sec-NCG} our first example of a non-commutative space, 
namely the ``quantum plane'', a one-parameter deformation of the ordinary complex plane,
and we extend certain basic operations from ordinary spaces to non-commutative ones.

In Sect.\,\ref{sec-Hopf} we consider the case when a space has an additional group structure. 
This naturally leads us to the notion of a Hopf algebra.
In Sect.\,\ref{ssec-group} we present two mutually dual Hopf algebras constructed from a finite group.

In~Sect.\,\ref{sec-SL2} we introduce two quantum groups associated with the Lie group $SL_2(\CC)$; 
one is its quantum coordinate algebra~$\SL_q(2)$, 
the other one is the quantum enveloping algebra~$U_q\, \gs\gl(2)$ of the Lie algebra of~$SL_2(\CC)$. 
We also construct a duality map between them and consider two interesting quotients.

In Sect.\,\ref{sec-group-action} we extend the notion of a group action to the non-commutative world. 
This leads us to the concept of a comodule algebra over a Hopf algebra. 
We give various examples of comodule algebras, thus showing that this concept covers much more than just group actions. 
In particular, any group-graded algebra is a comodule algebra over a suitable Hopf algebra. 
We also show how to equip the quantum plane with the structure of a comodule algebra 
over the quantum coordinate algebra of~$SL_2(\CC)$.

Section\,\ref{sec-HopfGalois} is entirely devoted to Hopf Galois extensions, 
which are non-commu\-ta\-tive analogues of principal fiber bundles. 
We pose the problem of classifying them and show that, contrary to the classical case, 
there may exist (infinitely many) non-isomorphic non-commutative principal fiber bundles over a point.
We also define the non-commutative version of the pull-back of a bundle.

In the final section (Sect.\,\ref{sec-versal}), for any Hopf algebra~$H$
we construct a non-commu\-ta\-tive principal fiber bundle in the form of a deformation~$\AA_H$ of~$H$ 
over a parameter space~$\BB_H$ which is the coordinate algebra of a smooth affine algebraic variety 
of the same dimension as~$H$.
We give explicit formulas for this non-commutative  principal fiber bundle 
when $H$ is the quantum enveloping algebra~$U_q\, \gs\gl(2)$ or some of its finite-dimensional quotients.

We will not give the proofs of all statements in these notes. 
For some of them we will refer to the relevant publications or to exercises if they turn out to be rather simple.
Except for Theo\-rems\,\ref{prop-def-Uq} and\,\ref{prop-def-ud} in Section\,\ref{sec-deform-Uq}, the material
presented in these notes already exists in the literature.

\section{Review of principal fiber bundles}\label{sec-bundle}

\begin{flushright}
\emph{La geometr\'\i a fue m\'as compasiva tal vez por \\
obra y gracia de su prestigio literario.} \cite{GM}\\
\end{flushright}

We start by recalling the definition and the basic properties of fiber bundles and of principal fiber bundles. 
In~Sect.\,\ref{sec-HopfGalois} we will define non-commutative analogues of such bundles.

\subsection{Fiber bundles}\label{ssec-fiber}

Let $F$ be a topological space. Recall that a \emph{fiber bundle}\index{fiber bundle} with fiber~$F$ 
is a locally trivial continuous map $\pi: P \to X$
from a topological space~$P$, called the \emph{total space}\index{fiber bundle!total space} of the bundle, 
to a topological space~$X$, called the \emph{base space}\index{fiber bundle!base space of},
such that each \emph{fiber}\index{fiber bundle!fiber of}~$\pi^{-1}(\{x\})$ is homeomorphic to~$F$.
\emph{Locally trivial}\index{fiber bundle!locally trivial} means that for each $x\in X$ 
there is a neighbourhood~$U \subset X$ of~$x$ and a homeomorphism
$\psi: \pi^{-1}(U) \cong U \times F$ such that $\pi = p_1 \circ \psi$, 
where $p_1: U\times F \to U$ is the first projection onto~$U$.

In the sequel we assume that the topological spaces we consider are Hausdorff and paracompact 
(the latter means that every open cover has a locally finite open refinement). 
These conditions are satisfied by most spaces generally considered.

A \emph{fiber bundle map}\index{fiber bundle!map of} from a fiber bundle $\pi': P' \to X'$ 
to another fiber bundle $\pi: P \to X$ with the same fiber~$F$ 
is a pair $(\widetilde{\varphi}: P' \to P, \,\varphi: X' \to X)$ of continuous maps such that
$\pi \circ \widetilde{\varphi} = \varphi \circ \pi'$.
The composition of two such maps is again a fiber bundle map.
A fiber bundle map is said to be a \emph{homeomorphism of fiber bundles} if both $\widetilde{\varphi}: P' \to P$ and $\varphi: X' \to X$
are homeomorphisms.

The simplest example of a fiber bundle with fiber~$F$ and base space~$X$ is given by
the first projection $p_1 : X \times F \to X$. Any fiber bundle homeomorphic to such a fiber bundle is called
a \emph{trivial fiber bundle}\index{fiber bundle!trivial}.

\subsection{Pull-backs}\label{ssec-pullback}

We now deal with an important functoriality property.
Any fiber bundle $\pi: P \to X$ with fiber~$F$ and base space~$X$ together 
with any continuous map $\varphi: X' \to X$ induces
a fiber bundle $\pi': \varphi^*(P) \to X'$ with the same fiber~$F$ and with base space~$X'$.
The space~$\varphi^*(P)$ is defined by 
\begin{equation*}
\varphi^*(P) = \left\{ (x',p) \in X' \times P \; |\;  \varphi(x') = \pi(p) \right\} 
\end{equation*}
and the map $\pi': \varphi^*(P) \to X'$ is equal to the composite map $\varphi^*(P) \subset X' \times P \overset{p_1}{\to} X'$.
The fiber bundle $\pi': \varphi^*(P) \to X'$ is called the \emph{pull-back}\index{fiber bundle!pull-back of}\index{pull-back}
of the bundle $\pi: P \to X$
along the map $\varphi: X' \to X$.

Clearly, if $\varphi': X'' \to X'$ is another continuous map, then 
\[
\varphi'{}^* (\varphi^*(P)) \cong (\varphi \circ \varphi')^*(P).
\]

If $\id: X \to X$ is the identity map of~$X$, then $\id^*(P) = P$. 
It follows that any homeo\-morphism $\varphi: X' \to X$ induces a homeomorphism $\varphi^*(P) \cong P$.

\begin{exercise}\label{exo-bundle}
(a) Let $\pi: P \to X$ be a fiber bundle. Prove that if $i : \{x\} \to X$ is the inclusion of a point $x$ in~$X$,
then $i^*(P) = \pi^{-1}(\{x\})$ is the fiber of the bundle at~$x$.

(b) Show that any fiber bundle with base space equal to a point is trivial.

(c) Prove that the pull-back of a trivial fiber bundle is trivial.

(d) Let $X$ be a \emph{contractible} space, i.e. such that there is an element $x_0 \in X$ and 
a continuous map $\eta : X \times [0,1] \to X$ such that $\eta(x,0) = x$ and $\eta(x,1) = x_0$ for all $x\in X$.
Show that any fiber bundle with base space~$X$ is trivial.
\end{exercise}

For more on fiber bundles, see the classical references~\cite{Hu, St}.

\subsection{Principal fiber bundles}\label{ssec-ppal}\index{principal fiber bundle}\index{fiber bundle!principal}

We fix now a topological group~$G$. 

\begin{definition}
A \emph{principal $G$-bundle} is a fiber bundle $\pi: P \to X$ with a continuous left action $G \times P \to P$
satisfying the following two conditions:
\begin{enumerate}
\item[(i)]
we have $\pi(gp) = \pi(p)$ for all $g\in G$ and $p\in P$,

\item[(ii)]
for all $p,p' \in P$ with $\pi(p) = \pi(p')$ there is a unique element $g\in G$ such that $gp = p'$.
\end{enumerate}
\end{definition}

In other words, in a principal $G$-bundle the group action preserves each fiber $\pi^{-1}(x)$ and
the action of~$G$ on each fiber is free and transitive. It follows that each fiber is in bijection with~$G$
and that the space of orbits $G\backslash P$ is homeomorphic to the base space~$X$.

An equivalent way to express Conditions\,(i) and\,(ii) above is to require that
the map
\begin{equation}\label{eq-gamma}
\gamma: G \times P \to P \times P \, ; \;\; (g,p) \mapsto (gp,p)
\end{equation}
is a bijection from~$G \times P$ onto the subspace 
\begin{equation*}
P \times_X P = \left\{ (p,p') \in P \times P \; |\;  \pi(p) = \pi(p') \right\}.
\end{equation*}

Given principal $G$-bundles $\pi: P' \to X'$ and $\pi: P \to X$, 
a \emph{map of principal $G$-bundles}\index{principal fiber bundle!map of}
from the first one to the second one is a fiber bundle map $(\widetilde{\varphi}, \varphi)$ compatible
with the $G$-action, i.e. such that
$\widetilde{\varphi}(gp') = g\widetilde{\varphi}(p')$ for all $g\in G$ and $p' \in P'$.

\begin{example}
Given a topological space~$X$, let $G$ act on $P = G \times X$ by $g'(g,x) = (g'g,x)$ ($g,g' \in G$, $x\in X$).
This is a principal $G$-bundle. Any principal $G$-bundle homeomorphic to such a bundle is called 
a \emph{trivial principal $G$-bundle}\index{principal fiber bundle!trivial}.
\end{example}

\begin{example}\label{exa-ppal}
Consider the group~$S^1$ of complex numbers of modulus one.
Given an integer~$n \geq 1$, the map
$\pi_n: S^1 \to S^1$ defined by $\pi_n(z) = z^n$ is a principal $G$-bundle, 
where $G$ is the cyclic group~$\ZZZ/n$ of order~$n$.
\end{example}

\begin{exercise}
Prove that the principal $\ZZZ/n$-bundle $\pi_n: S^1 \to S^1$ of Example\,\ref{exa-ppal} 
is trivial if and only if $n = 1$.
\end{exercise}

\subsection{Functoriality and classification}\label{ssec-functor}

We now record important properties of principal $G$-bundles. 
For the proofs we refer to\,\cite[Chap.\,4]{Hu} or to\,\cite{St}.

\begin{theorem}\label{th-ppal}
(a) If $\pi: P \to X$ is a principal $G$-bundle and $\varphi: X' \to X$ is a continuous map,
then the pull-back\index{principal fiber bundle!pull-back of}\index{pull-back} 
$\pi': \varphi^*(P) \to X'$ is a principal $G$-bundle.

(b) If $\pi: P \to X$ is a principal $G$-bundle
and $\varphi_0, \varphi_1 : X' \to X$ are homotopic\footnote{That is, there exists a continuous map
$\Phi: X' \times [0,1] \to X$ such that $\Phi(x,0) = \varphi_0(x)$ and $\Phi(x,1) = \varphi_1(x)$ for all $x\in X'$.} 
continuous maps, then the principal $G$-bundles $\varphi_0^*(P)$ and $\varphi_1^*(P)$ are homeomorphic.

(c) There exists a principal $G$-bundle $\pi_G: EG \to BG$ such that for any principal $G$-bundle $\pi: P \to X$
there is a continuous map $\varphi: X \to BG$ such that $\varphi^*(EG)$ is homeomorphic to~$\pi: P \to X$;
the map~$\varphi$ is unique up to homotopy.
\end{theorem}

The base space of the principal $G$-bundle $\pi_G: EG \to BG$ is called the \emph{classifying space}\index{classifying space}
of the group~$G$. The terminology is justified by the following immediate consequence of the theorem.

\begin{corollary}\label{coro-ppal}\index{classification}
The map $\varphi \mapsto \varphi^*(EG)$ induces a bijection between the set~$[X,BG]$
of homotopy classes of continuous maps from~$X$ to~$BG$
and the set $\Iso_G(X)$ of homeomorphism classes of principal $G$-bundles with base space~$X$:
\begin{equation*}
[X,BG] \cong \Iso_G(X) .
\end{equation*}
\end{corollary}

Starting from the next section, we shall build up the algebraic language necessary to define 
non-commutative analogues of principal fiber bundles.

\section{Basic ideas of non-commutative geometry}\label{sec-NCG}

As we stated in the introduction, non-commutative geometry is based on the idea of 
(a)~replacing a space~$X$ by its (commutative) function algebra~$\OO(X)$, 
(b)~passing from commutative algebras to non-commutative algebras.
In this section we start with two simple geometric situations,
namely when $X$ is a finite set and when it is an affine algebraic variety. 
In~Sect.\,\ref{ssec-NCA} we present our first elementary example of a non-commutative space, 
namely the quantum plane, 
and in\,Sect.\,\ref{ssec-operations} we extend certain basic operations from spaces to non-commutative ones.

For \emph{deformation quantization}\index{quantization!deformation}, which is another way, 
inspired by quantum mechanics, to pass from commutative algebras to non-commutative algebras
see the lectures\,\cite{Gu} by Simone Gutt.

\subsection{Two classical dualities between spaces and algebras}\label{ssec-dual}

Let us now present two well-known correspondences between spaces and algebras. 
All algebras\index{algebra} we consider in these notes are $\CC$-algebras 
(i.e. defined over the field~$\CC$ of complex numbers).
We furthermore assume that all algebras are associative\index{algebra!associative} and unital. 
We denote the unit\index{unit} of an algebra~$A$ by~$1$, or by~$1_A$ to avoid any confusion.

\subsubsection{Finite sets}\label{sssec-set}
In the first example, the spaces which we consider are merely sets, or if one prefers, discrete topological spaces.
To any set~$X$ we associate its \emph{function algebra}\index{function algebra}~$\OO(X)$, which consists of all
complex-valued functions on~$X$. Given two such functions $u_1,u_2 : X \to \CC$, we may consider
any linear combination $\lambda_1 u_1 + \lambda_2 u_2$, where $\lambda_1$ and $\lambda_2$ are complex numbers;
the function $\lambda_1 u_1 + \lambda_2 u_2$ is defined by
\[
(\lambda_1 u_1 + \lambda_2 u_2)(x) = \lambda_1 u_1(x) + \lambda_2 u_2(x)
\]
for all $x\in X$. Similarly, the product\index{product} $u_1u_2$ of two functions $u_1, u_2 \in \OO(X)$ is defined 
by $(u_1u_2)(x) = u_1(x) u_2(x)$ for all $x\in X$. 
These operations provide~$\OO(X)$ with the structure of a commutative associative and unital $\CC$-algebra.
The unit\index{unit} is the constant function whose values are all equal to~$1$.

For any $x\in X$, consider the $\delta$-function\index{$\delta$-function} $\delta_x$ defined for all $y\in X$ 
by $\delta_x(y) = \delta_{x,y}$, where $\delta_{x,y}$ is the Kronecker symbol\footnote{Recall 
that $\delta_{x,y} = 1$ if $x = y$ and $\delta_{x,y} = 0$ otherwise.}.
The product of two $\delta$-functions is clearly given by
\begin{equation*}
\delta_x \, \delta_y = \delta_{x,y} \, \delta_x.
\end{equation*}
This means that each $\delta$-function is an idempotent, i.e., $\delta_x^2 = \delta_x$, 
and that the product of two distinct $\delta$-functions is zero.

If the set~$X$ is \emph{finite}, then the set~$\{\delta_x\}_{x\in X}$ of $\delta$-functions forms a basis of~$\OO(X)$ 
considered as a vector space over the complex numbers. 
Indeed, we can expand any function $u: X \to \CC$ in the following unique way:
\begin{equation*}
u = \sum_{x \in X} \, u(x) \, \delta_x .
\end{equation*}
Note that the unit of~$\OO(X)$ is the sum of the $\delta$-functions:
$1 = \sum_{x\in X} \, \delta_x$.

If the set~$X$ is of cardinality~$N$, we can order the elements of~$X$ and assume that $X = \{ x_1, \ldots, x_N\}$.
Consider the linear map 
\[
u \in \OO(X) \mapsto \left( u(x_1), \dots u(x_N) \right) \in \CC^N. 
\]
This map is clearly an isomorphism from 
$\OO(X)$ onto the $N$-dimensional vector space~$\CC^N$. It is also an algebra isomorphism
if we endow~$\CC^N$ with the product\index{product}
\begin{equation*}
(x_1, \ldots x_N) (y_1, \ldots y_N) = (x_1y_1, \ldots x_Ny_N).
\end{equation*}
In particular, the dimension of~$\OO(X)$ is equal to the cardinality of~$X$. Since a finite set
is determined up to bijection by its cardinality, it follows that a finite set~$X$ can be recovered
(up to bijection) from its function algebra~$\OO(X)$.

\subsubsection{Algebraic varieties}\label{sssec-algvar}

The next correspondence is more substantial, namely the one between algebraic varieties and commutative algebras.
Recall that a \emph{complex algebraic variety}\index{algebraic variety} is the set of solutions of a system 
of polynomial equations over the complex numbers:
more precisely, let $\Sigma$ be a set of polynomials in $\CC[X_1, \ldots, X_n]$; then the corresponding algebraic variety
is given by
\begin{equation*}
V = \left\{ \, (x_1, \ldots, x_n) \in \CC^n \, | \; P(x_1, \ldots, x_n) = 0 \;\; \text{for all} \; P\in \Sigma \, \right\} .
\end{equation*}
To $V$ we associate the quotient-algebra 
\begin{equation*}
\OO(V) = \CC[X_1, \ldots, X_n]/I_{\Sigma},
\end{equation*}
where $I_{\Sigma}$ is the ideal of~$\CC[X_1, \ldots, X_n]$ generated by~$\Sigma$. 
We say that $\OO(V)$ is the \emph{coordinate algebra}\index{coordinate algebra} of the algebraic variety~$V$.
The algebra~$\OO(V)$ is a finitely generated commutative $\CC$-algebra.

Conversely, let us start from a finitely generated commutative $\CC$-algebra~$A$. It can be written as the 
quotient of a polynomial algebras with finitely many variables, i.e. it is of the form
\begin{equation*}
A = \CC[X_1, \ldots, X_n]/I
\end{equation*}
for some ideal~$I \subset \CC[X_1, \ldots, X_n]$.
Then $A = \OO(V)$, where $V$ is the set of points $(x_1, \ldots, x_n) \in \CC^n$ satisfying the system of polynomial equations
$P(x_1, \ldots, x_n) = 0$ for all $P \in I$.

There is another way to find $V$ such that $A = \OO(V)$ for a given finitely generated commutative $\CC$-algebra~$A$.
Namely consider the set $\Alg(A,\CC)$ of characters of~$A$.
A~\emph{character}\index{character of an algebra} of~$A$ is an algebra homomorphism~$\chi$ from $A$ to~$\CC$, 
i.e.\ a linear form satisfying the conditions
\begin{equation*}
\chi(ab) = \chi(a) \chi(b) \quad\text{and} \quad \chi(1) = 1.
\end{equation*}
Now, if $A = \CC[X_1, \ldots, X_n]/I$, then a character $\chi: A \to \CC$ is determined by its values
$\chi(X_i) = x_i \in \CC$ on the generators $X_1, \ldots, X_n$. Since $\chi$ must be zero on the ideal~$I$,
this means that the $n$-tuple $(x_1, \ldots, x_n) \in \CC^n$ of  values must be a solution of the equations 
$P(x_1, \ldots, x_n) = 0$ for all $P \in I$. Such solutions form an algebraic variety~$V$ and we have $A = \OO(V)$.

Let us also observe that the characters of a finitely generated commutative $\CC$-algebra~$A$
are in bijection with its maximal ideals. Indeed, start from a character $\chi: A \to \CC$;
its kernel $\mm$ is an ideal of~$A$. Since $\chi$ is surjective, we have $A/\mm \cong \CC$ by
Noether's first isomorphism theorem. Therefore, $\mm$ is a maximal ideal.
Conversely, let $\mm$ be a maximal ideal of~$A$. 
Then $A/\mm$ is a field which is isomorphic to~$\CC$ by Zarisky's lemma or by Hilbert's Nullstellensatz. 
The composed algebra map $\chi: A \to A/\mm \cong \CC$ is a character of~$A$.

Let us now give some elementary examples of commutative algebras corresponding to algebraic varieties.

\begin{example}
The coordinate algebra\index{coordinate algebra} of a point is~$\CC$ since $\Alg(\CC,\CC)$ consists only of one element,
namely the identity map. This follows also from the description of the function algebra of a finite set
given in Sect.\,\ref{sssec-set}.
\end{example}

\begin{example}
The one-variable polynomial algebra $\CC[X]$ is the coordinate algebra of the \emph{complex line}~$\CC$ since any 
algebra homomorphism $\CC[X] \to \CC$ is determined by its value on the variable~$X$; 
equivalently, $\Alg(\CC[X],\CC) \cong \CC$.

Similarly, the two-variable polynomial algebra $\CC[X,Y]$ is the coordinate algebra of the \emph{complex plane}~$\CC^2$:
any algebra homomorphism $\CC[X,Y] \to \CC$ is determined by its values on~$X$ and~$Y$. 
We have $\Alg(\CC[X,Y], \CC) \cong \CC^2$.
\end{example}

\begin{example}
Let us now consider the algebra $A = \CC[X,X^{-1}]$ of Laurent polynomials in the variable~$X$.
Since $XX^{-1} = 1$, this algebra can also be seen as the quotient-algebra $\CC[X,Y]/(XY-1)$.
Here also any algebra homomorphism $\chi: A \to \CC$ is determined by its value~$\chi(X) = x \in \CC$
on the variable~$X$, but contrary to the case of~$\CC[X]$, 
the fact that $X$ is invertible in~$A$ puts the following restriction on~$x$, namely 
\[
x \, \chi(X^{-1}) = \chi(X) \, \chi(X^{-1}) = \chi(XX^{-1}) = \chi(1) = 1.
\] 
Therefore, $x$ is invertible in the field~$\CC$,
which is equivalent to $x \neq 0$. We deduce $\Alg(A,\CC) \cong \CC^{\times}$,
where $\CC^{\times} = \CC \setminus\{0\}$.
In other words, the algebra $\CC[X,X^{-1}]$ of Laurent polynomials is the coordinate algebra of 
the \emph{once-punctured complex line}.
\end{example}

\begin{example}
The algebra $\CC[X,Y]/(Y^2 - X^3 + X - 1)$ is the coordinate algebra of the \emph{elliptic curve} consisting of the points
$(x,y) \in \CC^2$ satisfying the equation 
\[
y^2 = x^3 - x +1.
\]
\end{example}

\begin{example}
Let $x_1, \ldots, x_N$ be distinct points in the complex line~$\CC$. Consider the quotient-algebra
$A = \CC[X]/(X - x_1, \ldots, X - x_n)$. Since the polynomials $X - x_i$ are coprime, we also have
$A = \CC[X]/(P)$, where $P$ is the degree~$N$ polynomial 
\[
P = (X - x_1) \cdots (X - x_n). 
\]
The assignment
$Q \in \CC[X] \mapsto (Q(x_1), \ldots, Q(x_N)) \in \CC^N$ induces an algebra isomorphism $A \cong \CC^N$.
This example shows that a finite set can be seen as a special case of an algebraic variety.
\end{example}

\subsection{Non-commutative algebras}\label{ssec-NCA}\index{non-commutative algebra}
From now on we deal with non-necessarily com\-mu\-ta\-tive algebras. 
We recall that all algebras we consider are associative unital $\CC$-algebras.

\subsubsection{Non-commutative polynomials}\label{ssec-NCpoly}\index{non-commutative polynomial}
The prototype of a finitely generated complex commutative algebra is 
the algebra of polynomials $\CC[X_1, \ldots, X_n]$ in finitely many variables.
In an analogous way
the prototype of a finitely generated not necessarily commutative complex algebra is the algebra
$\CC\, \langle X_1, \ldots, X_n\rangle$ of polynomials in $n$~\emph{non-commuting variables} $X_1, \ldots, X_n$.
Any element of~$\CC\, \langle X_1, \ldots, X_n\rangle$ is a finite linear combination (with complex coefficients)
of finite words in the letters $X_1, \ldots, X_n$. 
Such a linear combination is unique because such words form a basis of~$\CC \,\langle X_1, \ldots, X_n\rangle$
considered as a vector space over the complex numbers.

Mind the difference between these two kinds of polynomial algebras:
the element $XY -YX$ is non-zero in~$\CC\, \langle X, Y \rangle$ whereas it vanishes in~$\CC[X,Y]$.

Any finitely generated complex algebra~$A$ is a quotient-algebra of $\CC \, \langle X_1, \ldots, X_n\rangle$
for some~$n$, which means that $A$ can be expressed as
\[
A = \CC\,  \langle X_1, \ldots, X_n\rangle/I
\] 
for some two-sided ideal~$I$ of~$\CC \, \langle X_1, \ldots, X_n\rangle$. 
For instance, for the algebra of ordinary polynomials in $n$~variables, we have
\begin{equation*}\label{eq-plane}
\CC[X_1, \ldots, X_n] = \CC \, \langle X_1, \ldots, X_n\rangle/I,
\end{equation*}
where $I$ is the two-sided ideal generated by all elements of the form $X_iX_j - X_jX_i$ ($i, j \in \{1, \ldots, n\}^2$).

\subsubsection{The quantum plane}\label{sssec-qplane}\index{quantum plane}

Let $q$ be a non-zero complex number. 
Consider the algebra~$\CC \, \langle X, Y \rangle$ of polynomials in two non-commuting variables $X,Y$
and the two-sided ideal $I_q$ of~$\CC \, \langle X, Y \rangle$
generated by $YX - q XY$. 
The quotient-algebra\index{$\CC_q[X,Y]$}
\begin{equation*}
\CC_q[X,Y] = \CC \, \langle X, Y \rangle/I_q
\end{equation*}
is \emph{not commutative} unless $q = 1$. 

When $q = 1$, then the algebra~$\CC_q[X,Y]$ is isomorphic to~$\CC[X,Y]$,
which is the coordinate algebra of the plane. 
Thus, $\CC_q[X,Y]$ is a one-parameter non-commutative deformation\index{$q$-deformation}\index{deformation}
(or a quantization\index{quantization}) of the coordinate algebra of the plane. For this reason
and by extension, $\CC_q[X,Y]$ can be considered as the coordinate algebra of a ``space'' in an extended sense,
of a so-called \emph{non-commutative space}\index{non-commutative space}. 
In this particular instance, this non-commutative space is known in the literature under the name \emph{quantum plane}.

The set $\{X^i Y^j\}_{i,j \geq 0}$ forms a basis of~$\CC_q[X,Y]$, independently of~$q$
(see Exercise\,\ref{exo-qplane-basis} below). 
Notice that the defining relation $YX  = q XY$
implies the following product\index{product} formula for two monomials in~$\CC_q[X,Y]$:
\begin{equation*}
(X^i Y^j) (X^k Y^{\ell}) = q^{jk} \, X^{i+k} Y^{j + \ell}. \qquad\qquad (i,j, k,\ell \geq 0)
\end{equation*}

In Sect.\,\ref{sssec-algvar} we showed how to recover an algebraic variety~$V$ from its coordinate algebra, using its
characters. Let us look at the set $\Alg(\CC_q[X,Y],\CC)$ of characters\index{character of an algebra} of~$\CC_q[X,Y]$. 
As with the usual polynomial algebra~$\CC[X,Y]$,
a character $\chi: \CC_q[X,Y] \to \CC$ is determined by its values $\chi(X) = x$ and $\chi(Y) = y$ on the generators $X$ and~$Y$. 
Now the set $\Alg(\CC_q[X,Y],\CC)$ is in bijection with the set of points $(x,y) \in \CC^2$ such that $yx = qxy$.
In~$\CC$ the values $x$ and $y$ commute, so that $yx = qxy$ is equivalent to $(q-1)xy = 0$. 
When $q \neq 1$, then $\Alg(\CC_q[X,Y],\CC)$ can be identified with the subset of~$\CC^2$ defined by $xy = 0$;
this subset is the union of the lines $L_1 = \{0\} \times \CC$ and 
$L_2 = \CC \times \{0\} \subset \CC^2$. The coordinate algebra of $L_1 \cup L_2$ is 
the commutative algebra $\CC[X,Y]/(XY)$. We thus have bijections
\begin{equation*}
\Alg(\CC_q[X,Y],\CC) =
\begin{cases}
\Alg(\CC[X,Y],\CC)  = \CC^2 \qquad\qquad\quad \;\; \text{if} \; q = 1, \\
\noalign{\smallskip}
\Alg(\CC[X,Y]/(XY),\CC) = L_1 \cup L_2\quad \text{if} \; q \neq 1.
\end{cases}
\end{equation*}
This shows that from the point of view of characters, there is a jump when we pass from $q = 1$
to an arbitrary complex number~$q$.
Observe also that as a vector space, $\CC[X,Y]/(XY)$ has a basis given by $\{X^i\}_{i\geq 0} \cup \{Y^j\}_{j\geq 1}$;
this basis is clearly very different from the basis $\{X^i Y^j\}_{i,j \geq 0}$ of~$\CC_q[X,Y]$.

\begin{exercise}\label{exo-qplane-basis}
\emph{(A basis of the quantum plane)}

(a) Let $\tau$ and $\upsilon$ be the endomorphisms of the polynomial algebra~$\CC[t]$ defined on any
polynomial~$P(t)$ by 
$\tau(P(t)) = tP(t)$ and $\upsilon(P(t)) = P(qt)$.
Show that there is a unique algebra morphism $\rho: \CC_q[X,Y] \to \End(\CC[t])$ such that
$\rho(X) = \tau$ and $\rho(Y) = \upsilon$.

(b) Deduce that $\{X^i Y^j\}_{i,j \in \NN}$ is a basis of~$\CC_q[X,Y] $. 
Hint: use the morphism~$\rho$ to prove linear independence.
\end{exercise}

\subsubsection{Non-commutative spaces}\label{ssec-NCspace}

In view of the previous examples, non-com\-mu\-ta\-tive algebras\index{non-commutative algebra} 
will henceforth often be called \emph{non-commutative spaces}\index{non-commutative space}. 
The special case of the quantum plane shows that characters are not sufficient to characterize non-com\-mu\-ta\-tive spaces.
As written in the introduction of\,\cite{Pf},
\begin{quote}
... in noncommutative geometry there are no points.
\end{quote}
This is a significant difference with ordinary spaces.
Such a difference is also well explained in\,\cite[Sect.\,2]{Si}.

\subsection{Extending basic operations to non-commutative spaces}\label{ssec-operations}

We now show how to extend certain basic operations on spaces to the world of non-commutative spaces,
i.e. of non-necessarily commutative algebras.

\subsubsection{From maps to algebra homomorphisms}\label{ssec-NCmaps}

Let $\varphi: X\to Y$ be a map between algebraic varieties. 
Then we can define a map $\varphi^* : \OO(Y) \to \OO(X)$ by
\begin{equation}\label{def-pullback}
\varphi^*(u) = u \circ \varphi
\end{equation}
for all $u\in \OO(Y)$.
It is easy to check that $\varphi^*$ is a morphism of algebras.

If $\psi: Y \to Z$ is another map between algebraic varieties and $\psi^* : \OO(Z) \to \OO(Y)$ 
is the corresponding morphism of algebras, 
then we have the following equality of morphisms from~$\OO(Z)$
to~$\OO(X)$:
\begin{equation*}
(\psi \circ \varphi)^* = \varphi^* \circ \psi^*.
\end{equation*}

\subsubsection{From products to tensor products}\label{ssec-NCproduct}

Given algebraic varieties $X$, $Y$, we can consider their product $X \times Y$.
We denote by $\pi_X : X \times Y \to X$ and $\pi_Y : X \times Y \to Y$ the canonical projections.
The product $X \times Y$ satisfies the following universal property:
for all maps $\varphi_X: Z \to X$ and $\varphi_Y: Z \to Y$ from another algebraic variety~$Z$,
there exists a unique map $\varphi: Z \to X \times Y$ such that $\pi_X \circ \varphi = \varphi_X$
and $\pi_Y \circ \varphi = \varphi_Y$.

Applying the contravariant functor $\varphi \mapsto \varphi^*$ defined by\,\eqref{def-pullback}, we see that 
the coordinate algebra $\OO(X \times Y)$ comes with two algebra morphisms 
\begin{equation*}
\varphi_X^*: \OO(X) \to \OO(X \times Y) \quad\text{and}\quad
\varphi_Y^*: \OO(Y) \to \OO(X \times Y)
\end{equation*}
satisfying a universal property that is easily deduced from the universal property of the product~$X\times Y$.
It follows that we have a canonical algebra isomorphism
\begin{equation}\label{eq-tensor}
\OO(X \times Y) \cong \OO(X) \otimes \OO(Y),
\end{equation}
where $\OO(X) \otimes \OO(Y)$  is the tensor product of the algebras $\OO(X)$ and~$\OO(Y)$.

Let us recall that the \emph{tensor product}\index{tensor product of vector spaces} $U \otimes V$ 
of two complex vector spaces $U$ and~$V$ consists
of $\CC$-linear combinations of symbols of the form $u\otimes v$, where $u\in U$ and $v\in V$.
By definition, the map $U\times V \to U\otimes V$ sending each couple $(u,v) \in U\times V$ to~$u\otimes v$
is $\CC$-bilinear, i.e. $\CC$-linear both in~$u$ and in~$v$. 
It satisfies the following universal property: for any $\CC$-bilinear map $f: U \times V \to W$ to another vector space~$W$,
there is a unique $\CC$-linear map $\widetilde{f}: UÊ\otimes V \to W$ such that $f(u,v) = \widetilde{f}(u\otimes v)$
for all $(u,v) \in U\times V$. 
Moreover, if $\{u_i \}_{i\in I}$ is a basis of~$U$ and $\{v_j \}_{j\in J}$ is a basis of~$V$,
then 
\[
\{u_i\otimes v_j \}_{(i,j)\in I \times J}
\]
is a basis of~$U \otimes V$. As a consequence, $\dim(U \otimes V) = \dim(U) \, \dim(V)$.

If $A$, $B$ are (not necessarily commutative) algebras, then their tensor product\index{tensor product of algebras} 
$A\otimes B$ carries a structure of algebra with multiplication determined by
\begin{equation*}
(a_1 \otimes b_1) (a_2 \otimes b_2) = a_1a_2 \otimes b_1b_2
\end{equation*}
for all $a_1, a_2 \in A$ and $b_1, b_2 \in B$. The algebra $A\otimes B$ has a unit\index{unit} given by 
\begin{equation*}
1_{A\otimes B} = 1_A \otimes 1_B .
\end{equation*}

The tensor product of algebras satisfies the following universal property.

\begin{proposition}\label{prop-tens-alg}
Let $f: A\to C$ and $g : B\to C$ be morphisms of algebras such that $f(a) g(b) = g(b) f(a)$ in~$C$ for all $a\in A$ and $b\in B$.
Then there exists a unique morphism of algebras $f\otimes g: A\otimes B \to C$ 
such that $(f\otimes g)(a\otimes b) = f(a) g(b)$ for all $a\in A$ and $b\in B$.
\end{proposition} 

Using the notation $\Alg(A_1,A_2)$ for the set of morphisms of algebras from $A_1$ to~$A_2$,
we can paraphrase the previous proposition by saying that
$\Alg(A\otimes B,C)$ is isomorphic to the subset of $\Alg(A,C) \times \Alg(B,C)$ consisting of all pairs $(f,g)$
of morphisms whose images commute in~$C$. In particular, if $C$ is commutative, then
\begin{equation*}
\Alg(A\otimes B,C) \cong \Alg(A,C) \times \Alg(B,C) .
\end{equation*}

For this reason we may consider the tensor product of algebras as the non-commutative analogue of the product of spaces.

\begin{exercise}
Prove Proposition\,\ref{prop-tens-alg}.
\end{exercise}

\section{From groups to Hopf algebras}\label{sec-Hopf}

In this section we introduce the concept of a Hopf algebra and illustrate it 
with several examples which will show up repeatedly in these notes.
For general references on Hopf algebras, see\,\cite{Ab, Ka-QG, Mo, Sw}.

\subsection{Algebraic groups}\label{ssec-alggr}

Let $G$ be an \emph{algebraic group}\index{algebraic group}, 
i.e.\ an algebraic variety equipped with the structure of a group such that
the product map $\mu :G\times G \to G$ is a map of algebraic varieties.

The basic example of an algebraic group is the \emph{general linear group}\index{general linear group}~$GL_N(\CC)$,
which consists of all invertible $N \times N$-matrices with complex entries, equipped with the usual matrix product.
This product is given by polynomial formulas in the entries.
The coordinate algebra\index{coordinate algebra} of~$GL_N(\CC)$ is the algebra
\begin{equation}\label{def-OGL}
\OO(GL_N(\CC)) = \CC[t, (a_{i,j})_{1 \leq i, j\leq N}] / (t\det(a_{i,j}) -1). 
\end{equation}

Any subgroup of $GL_N(\CC)$ defined by the vanishing of polynomials is also an algebraic group. For instance,
the \emph{special linear group}\index{special linear group}~$SL_N(\CC)$, 
which consists of all $N \times N$-matrices whose determinant is~$1$,
is an algebraic group. Its coordinate algebra is the algebra
\begin{equation*}
\OO(SL_N(\CC)) = \CC[(a_{i,j})_{1 \leq i, j\leq N}] / (\det(a_{i,j}) -1). 
\end{equation*}
It is obtained from $\OO(GL_N(\CC))$ by setting $t = 1$.

By\,\eqref{def-pullback} the product\index{product} map $\mu :G\times G \to G$ 
of an algebraic group induces a morphism of algebras $\mu^* : \OO(G) \to \OO(G\times G)$. 
We can compose~$\mu^*$ with the canonical isomorphism $\OO(G \times G) \cong \OO(G) \otimes \OO(G)$
(see\,\eqref{eq-tensor}), which yields a morphism of algebras 
\begin{equation*}
\Delta: \OO(G) \to \OO(G) \otimes \OO(G),
\end{equation*}
which we call the \emph{coproduct}\index{coproduct} of~$\OO(G)$.

The product~$\mu$ of~$G$ is \emph{associative}, which means that we have
\begin{equation*}
\mu\left (\mu(g_1,g_2),g_3 \right) = \mu \left(g_1,\mu(g_2,g_3) \right)
\end{equation*}
for all $g_1,g_2,g_3 \in G$.
This identity, which reads 
$\mu\circ (\mu \otimes \id) = \mu\circ (\id\otimes \mu)$,
transposes to the following \emph{coassociativity}\index{coassociativity} identity for the coproduct:
\begin{equation}\label{eq-coassoc}
(\Delta \otimes \id) \circ \Delta = (\id \otimes \, \Delta) \circ \Delta.
\end{equation}

Similarly, the \emph{unit}\index{unit}~$e$ of the group~$G$, which can be seen as a homomorphism $\bar{e} : \{1\} \to G$
(sending~$1$ to~$e$), induces the morphism of algebras 
\begin{equation*}
\eps = \bar{e}^* : \OO(G) \to \OO(\{1\}) = \CC,
\end{equation*}
which we call the \emph{counit}\index{counit} of~$\OO(G)$.
The identities $\mu(e,g) = g = \mu(g,e)$ ($g\in G$) read 
\[
\mu\circ ( \bar{e}\otimes \id)= \id= \mu \circ (\id \otimes  \bar{e}),
\]
where we have identified  $\{1\}\times G$ and $G\times \{1\}$ with~$G$.
They transpose to the \emph{counitality}\index{counitality} identities
\begin{equation}\label{eq-counit}
(\eps \otimes \id) \circ \Delta = \id = (\id \otimes \, \eps) \circ \Delta : \OO(G) \to \OO(G),
\end{equation}
where we use the natural identifications $\CC \otimes \, \OO(G) \cong \OO(G)$ and $\OO(G) \otimes \, \CC \cong \OO(G)$.

In a group~$G$ any element~$g$ possesses an \emph{inverse}, i.e.\ an element~$g^{-1}$
such that 
\begin{equation}\label{eq-inv}
\mu(g,g^{-1}) = e = \mu(g^{-1},g).
\end{equation}
The map $\inv: g \mapsto g^{-1}$ induces a map $S = \inv^* : \OO(G) \to \OO(G)$,
which we call the \emph{antipode}\index{antipode} of~$\OO(G)$. The identities\,\eqref{eq-inv} imply
identities for the antipode, which we shall display in Sect.\,\ref{ssec-Hopfalg}.

When $G = GL_N(\CC)$ is the general linear group, 
the coproduct of the coordinate algebra~$\OO(GL_N(\CC))$ 
is defined on the generators~$t$, $a_{i,j}$ of~$\OO(GL_N(\CC))$ by
\begin{equation}\label{eq-coprod-GL}
\Delta(t) = t \otimes t
\quad\text{and}\quad
\Delta(a_{i,j}) = \sum_{k=1}^N \, a_{i,k} \otimes a_{k,j}
\end{equation}
and the counit by 
\begin{equation}\label{eq-counit-GL}
\eps(t) = 1
\quad\text{and}\quad
\eps(a_{i,j}) = \delta_{i,j}
\end{equation}
for all $i,j \in \{1, \ldots, N\}$.
For the antipode, let~$A$ be the $N\times N$-matrix $A = (a_{i,j})_{1 \leq i, j\leq N}$.
Denote by~$A_{i,j}$ the determinant of the $(N-1) \times (N -1)$ matrix obtained from deleting Row~$i$ and Column~$j$
of~$A$. Then for each generator~$a_{i,j}$ ($i,j \in \{1, \ldots, N\}$) we have
\begin{equation}\label{eq-anti-GL}
S(a_{i,j}) = (-1)^{i+j}  \, \frac{A_{j,i}}{\det(A)} \, .
\end{equation}
By the definition\,\eqref{def-OGL} the generator~$t$ is invertible with inverse $t^{-1} = \det(A)$
and its antipode is given by~$S(t) = t^{-1} = \det(A)$.
 
The values of $\Delta(a_{i,j})$, $\eps(a_{i,j})$ and $S(a_{i,j})$ given in
Formulas\,\eqref{eq-coprod-GL}--\eqref{eq-anti-GL} above also 
determine the coproduct, counit and antipode of~$\OO(SL_N(\CC))$, where $SL_N(\CC)$ is the special linear group.

\begin{exercise}
Prove the claims of this section.
\end{exercise}

\subsection{Bialgebras}\label{ssec-bialg}

Before defining Hopf algebras, we present the concept of a bialgebra.

\begin{definition}\label{def-bialgebra}
A \emph{bialgebra}\index{bialgebra} is an associative unital algebra equipped with two linear maps
$\Delta : H \to H \otimes H$ and $\eps: H \to \CC$
satisfying the following conditions:
\begin{enumerate}
\item[(i)]
The maps $\Delta$ and $\eps$ are morphisms of algebras.

\item[(ii)]
We have the following equalities:
\begin{equation}\label{eq-coass}\index{coassociativity}
(\Delta \otimes \id) \circ \Delta = (\id \otimes \, \Delta) \circ \Delta.
\end{equation}
and, identifying $\CC \otimes H$ and $H \otimes \CC$ with~$H$,
\begin{equation}\label{eq-coun}\index{counitality}
(\eps \otimes \id) \circ \Delta = \id = (\id \otimes \, \eps) \circ \Delta.
\end{equation}
\end{enumerate}
\end{definition}

The map $\Delta$ is called the \emph{coproduct}\index{coproduct} of~$H$ and $\eps$ is its \emph{counit}\index{counit}.
It is sometimes convenient to denote the product\index{product} of the bialgebra~$H$ by $\mu: H \otimes H \to H$ and 
to introduce the unique morphism of algebras $\eta: \CC \to H$, 
which we call the \emph{unit}\index{unit} of~$H$; we have $\eta(1) = 1_H$.

Given a bialgebra~$H$ with coproduct~$\Delta$, we define the \emph{opposite coproduct}\index{coproduct!opposite} 
\[
\Delta^{\op}: H \to H \otimes H
\]
by
$\Delta^{\op} = \tau \circ \Delta$, where $\tau: H \otimes H \to H \otimes H$ is the \emph{flip}
defined by $\tau(x \otimes y) = y \otimes x$ for all $x,y \in H$.
We say that $H$ is \emph{cocommutative}\index{Hopf algebra!cocommutative}\index{bialgebra!cocommutative} 
if $\Delta^{\op} = \Delta$.

\begin{exercise}\label{exo-bi-not-Hopf}
Let $\CC[t]$ be the polynomial algebra in one variable~$t$. Show that $\CC[t]$ is a bialgebra
with coproduct~$\Delta$ and counit~$\eps$ determined by $\Delta(t) = t\otimes t$ and $\eps(t) = 1$.
Check that this bialgebra is cocommutative.
\end{exercise}

\begin{exercise}\label{exo-dualHopf}
(a) Let $H$ be a bialgebra with coproduct $\Delta$ and counit~$\eps$.
Consider the \emph{linear dual}\index{bialgebra!dual of} $H\, \check{} = \Hom(H,\CC)$ of~$H$. 
Define a product\index{product}
$\mu\, \check{} : H\, \check{} \otimes H\, \check{} \to H\, \check{}\, $ 
for all $x\in H$ and $\alpha, \beta \in H\, \check{}$ by
\begin{equation}\label{dual-prod}
\mu\, \check{}\, (\alpha \otimes \beta)(x) = (\alpha \otimes \beta)(\Delta(x))
= \sum_i\, \alpha(x'_i) \, \beta(x''_i) , 
\end{equation}
when $\Delta(x) = \sum_i\, x'_i \otimes x''_i$. Show that 
\begin{enumerate}
\item[(i)]
$\mu\, \check{}\,$ is an associative product with unit\index{unit} equal to $\eps \in H\, \check{}$,

\item[(ii)]
$H\, \check{}\,$ is cocommutative if $H$ is a commutative algebra.
\end{enumerate}

(b) Now assume that $H$ is finite-dimensional as a vector space over~$\CC$. 
\begin{enumerate}
\item[(i)]
Show that
$H\, \check{}\, $ is a bialgebra with coproduct $\Delta\, \check{} : H\, \check{} \to H\, \check{} \otimes H\, \check{}\,$
and counit $\eps\, \check{}: H\, \check{} \to \CC$ defined by
\begin{equation*}
\Delta\, \check{}\, (\alpha)(x\otimes y) = \alpha(xy) 
\end{equation*}
and $\eps\, \check{}\, (\alpha) = \alpha(1_H)$ for all $\alpha \in H\, \check{}$.

\item[(ii)] 
Prove that $H\, \check{}\, $ is commutative if $H$ is cocommutative.
\end{enumerate}
\end{exercise}

\begin{remark}\label{rem-restricted}
It follows from Exercise\,\ref{exo-dualHopf} that the dual of a finite-dimensional bialgebra is another (finite-dimensional) bialgebra. 
To extend such a duality to the case when $H$ is an infinite-dimensional bialgebra, 
we have to replace the linear dual $H\, \check{}\,$ 
by the \emph{restricted dual}\index{Hopf algebra!restricted dual of}~$H^{\circ}$ defined by
\begin{equation*}
H^{\circ} = \left\{ \alpha \in H\, \check{} \; |\, \alpha(I) = 0 \;\textrm{for some ideal} \; I \; \textrm{such that} \;
\dim\,H/I < \infty \right\} .
\end{equation*}
See \cite[Sect.\,1.2]{Mo} or \cite{Sw}. We have $H^{\circ} = H\, \check{}\,$ if~$\dim\, H < \infty$.
\end{remark}

\subsection{Hopf algebras}\label{ssec-Hopfalg}

Let $H$ be a bialgebra with product $\mu$, unit~$\eta$, coproduct~$\Delta$, and counit~$\eps$.
Given two linear endomorphisms $f$, $g$ of~$H$ we define a new linear endomorphism $f*g$ of~$H$ by
\begin{equation}\label{def-convol}
f*g = \mu \circ (f\otimes g) \circ \Delta \in \End(H) .
\end{equation}

We now define the concept of a Hopf algebra.

\begin{definition}\label{def-Hopfalg}
Let $H$ be a bialgebra.

(a) An \emph{antipode}\index{antipode} of~$H$ is a linear endomorphism~$S$ of~$H$ such that
\begin{equation}\label{eq-anti}
S * \id_H = \eta \circ \eps = \id_H * S.
\end{equation}

(b) A \emph{Hopf algebra}\index{Hopf algebra} is a bialgebra together with an antipode.

(c) A \emph{morphism of Hopf algebras}\index{Hopf algebra!morphism of} $f: H \to H'$ between Hopf algebras 
is a morphim of bialgebras such that 
\[
\Delta' \circ f = (f \otimes f) \circ \Delta, \quad 
\eps' \circ f = \eps,
\quad
S' \circ f = f \circ S, 
\]
where $\Delta$ (resp.~$\Delta'$) is the coproduct, $\eps$ (resp.~$\eps'$) is the counit
and $S$ (resp.~$S'$) is the antipode of~$H$ (resp.~of~$H'$).
\end{definition}

\begin{example}
If $G$ is an \emph{algebraic group}\index{algebraic group}, then its coordinate algebra\index{coordinate algebra}~$\OO(G)$
equip\-ped with the maps $\Delta$, $\eps$, and~$S$ defined in Sect.\,\ref{ssec-alggr}
is a Hopf algebra. Actually, the axioms of a Hopf algebra are derived from this example.
\end{example}

Hopf algebras have two important features which are worth emphasizing:
\begin{itemize}
\item 
The concept of Hopf algebras is \emph{self-dual}: 
the restricted dual~$H^{\circ}$ of a Hopf algebra~$H$ is again a Hopf algebra 
(see Exercises\,\ref{exo-dualHopf}\,(b) and\,\ref{exo-Hopf1} for finite-dimensional Hopf algebras). 
This duality allows also to extend the Pontryagin duality of abelian groups to non-abelian ones (see Exercise\,\ref{exo-duality}).

\item
The category of left $H$-modules, where $H$ is a Hopf algebra, is a \emph{tensor category}. 
Recall that a left $H$-module~$V$ is a vector space together with a bilinear map $H \times V \to V;
(x,v) \mapsto xv$ ($x, \in H, vÊ\in V$) such that
\begin{equation}\label{def-action}
(xy)v = x(y(v)) 
\quad\text{and}\quad
1_H v = v
\end{equation}
for all $x,y\in H$ and $v\in V$. The map $(x,v) \mapsto xv$ is called the action.

If $V_1$ and $V_2$ are left $H$-modules, then so is the tensor product $V_1 \otimes V_2$.
Indeed one defines an action of~$H$ on~$V_1 \otimes V_2$ by
\begin{equation}\label{act-tensor}
x(v_1 \otimes v_2) = \Delta(x) (v_1 \otimes v_2) = \sum_{i}\, x'_iv_1 \otimes x''_iv_2
\end{equation}
if $\Delta(x) = \sum_{i}\, x'_i \otimes x''_i$.

\begin{exercise}
Check that the action\,\eqref{act-tensor} of $H$ on~$V_1 \otimes V_2$ satisfies\,\eqref{def-action}.
\end{exercise}
\end{itemize}

\begin{remark}
In many cases, for instance when $H$ is a quantum group as in~Sect.\,\ref{sec-SL2},
$V_1 \otimes V_2$ is naturally isomorphic as an $H$-module to $V_2 \otimes V_1$. 
It is this feature that leads to braid group representations and knot invariants. 
We will not say more about this; see\,\cite[Part Three]{Ka-QG} for details on this vast subject.
\end{remark}

\begin{exercise}
Show that the product~$*$ on the algebra~$\End(H)$ of linear endomorphisms of a Hopf algebra~$H$ given by\,\eqref{def-convol}
is associative with unit\index{unit} equal to~$\eta \circ \eps$. Prove that an antipode is unique if it exists.
\end{exercise}

\begin{exercise}\label{exo-Hopf1}
Show that the dual\index{Hopf algebra!dual of} $H\, \check{}$ of a finite-dimensional Hopf algebra~$H$ is a Hopf algebra.
\end{exercise}

\begin{exercise}
\emph{(A bialgebra without antipode)}
Let $\CC[t]$ be the bialgebra considered in Exercise\,\ref{exo-bi-not-Hopf}.
Prove that it has no antipode [hint: apply\,\eqref{eq-anti} to the element~$t$].
\end{exercise}

The following properties of the antipode of a Hopf algebra are worth mentioning (see\,\cite[III.3]{Ka-QG} or\,\cite{Sw}).

\begin{proposition}\label{prop-anti}
Let $H$ be a Hopf algebra with coproduct~$\Delta$, counit~$\eps$, and antipode~$S$.

(a) The antipode $S$ is an anti-morphism of algebras, i.e., for all $x,y \in H$,
\begin{equation*}
S(xy) = S(y) S(x) \quad \textrm{and} \quad S(1) = 1,
\end{equation*}
and we have
\begin{equation*}
(S \otimes S) \circ \Delta = \Delta^{\op} \circ S 
\quad \textrm{and} \quad
\eps \circ S = \eps.
\end{equation*}

(b) If $H$ is commutative or cocommutative, then the antipode~$S$ is an involution, i.e. $S^2 = \id_H$.
\end{proposition}

Another useful concept is the following. An element $x$ of a Hopf algebra~$H$ is called 
\emph{group-like}\index{group-like element} if 
\begin{equation}\label{def-grouplike}
\Delta(x) = x \otimes x \quad \text{and} \quad \eps(x) = 1.
\end{equation}
Let $\Gr(H)$ be the set of group-like elements of~$H$.

\begin{proposition}\label{prop-GrH}
The set $\Gr(H)$\index{$\Gr(H)$} of group-like elements of~$H$ is a group under the product in~$H$. 
The inverse of an element~$x$ in~$\Gr(H)$ is~$S(x)$.
\end{proposition}

\begin{proof}
Let $x,y\in H$ be group-like elements. Since $\Delta$ and $\eps$ are morphisms of algebras, we have
\begin{equation*}
\Delta(xy) = \Delta(x) \Delta(y) = (x \otimes x) (y \otimes y ) = xy \otimes xy
\end{equation*}
and $\eps(xy) = \eps(x) \eps(y) = 1$. This shows that $\Gr(H)$ is preserved under the product.
Clearly, the unit~$1$ of~$H$ is group-like and is a unit for the product in~$\Gr(H)$.

Applying\,\eqref{eq-anti} to a group-like element~$x$, we obtain $S(x) x = 1 = x S(x)$,
which shows that $S(x)$ is the inverse of~$x$. To conclude that $\Gr(H)$ is a group, it remains to check that
$S(x)$ is group-like. Indeed, by Proposition\,\ref{prop-anti}\,(a),
\begin{equation*}
\Delta^{\op}(S(x)) = (S \otimes S)(\Delta(x)) = S(x) \otimes S(x),
\end{equation*}
which implies $\Delta(S(x)) = S(x) \otimes S(x)$. We also have 
$\eps(S(x)) = \eps(x) = 1$. Thus, $S(x)$ is group-like.
\qed
\end{proof}

Examples of group-like elements and computations of~$\Gr(H)$ will be given in Exercise\,\ref{exo-GrH} below.

\subsection{Examples of Hopf algebras from finite groups}\label{ssec-group} 

To familiarize the reader with the concept of a Hopf algebra, we now present the following two basic examples,
both constructed from a group.

\subsubsection{The function algebra of a finite group}\label{sssec-OG}\index{function algebra}

Let $G$ be a finite group with unit~$e$ 
and $\OO(G)$ be its function algebra, as defined in Sect.\,\ref{sssec-set}.
It is a Hopf algebra with coproduct~$\Delta$, counit~$\eps$, and antipode~$S$ given by
\begin{equation}\label{eq-defOG}
\Delta(u)(g,h) = u(gh), \quad\;\; \eps(u) = u(e), \quad\;\; S(u)(g) = u(g^{-1})
\end{equation}
for all $g,h \in G$ and $u \in \OO(G)$. Here we have identified $\OO(G) \otimes \OO(G)$
with the function algebra~$\OO(G\times G)$ of the product group $G\times G$.

We can also express $\Delta$, $\eps$, and~$S$ in terms of the $\delta$-functions\index{$\delta$-function} 
introduced in \emph{loc.\ cit.} Namely we have
\begin{equation*}
\Delta(\delta_g) = \sum_{h\in G} \, \delta_h \otimes \delta_{h^{-1}g}, \quad
S(\delta_g) = \delta_{g^{-1}}, \quad
\eps(\delta_g) =
\begin{cases}
1 & \text{if}\;  g = e,\\
0 & \text{otherwise.}
\end{cases}
\end{equation*}

Since the inverse map $g\mapsto g^{-1}$ in a group is an involution, it follows from\,\eqref{eq-defOG}
that the antipode~$S$ is an involution as well, which is in agreement with Proposition\,\ref{prop-anti}\,(b)
applied to the \emph{commutative}\index{algebra!commutative} Hopf algebra~$\OO(G)$.

\subsubsection{The convolution algebra of a group}\label{sssec-CG}

Let $G$ now be an arbitrary group, not necessarily finite. We define $\CC[G]$ to be the vector space
spanned by the elements of~$G$. This means that any element of~$\CC[G]$
is a linear combination of the form
\begin{equation*}
\sum_{g\in G}\, \lambda_g \, g,
\end{equation*}
where the coefficients~$\lambda_g$ are complex numbers, all of which are zero except for a finite number.
We also assume that the set~$\{g\}_{g\in G}$ is a basis of~$\CC[G]$, which is equivalent to the implication
\begin{equation*}
\left(\sum_{g\in G}\, \lambda_g \, g = 0 \right)\; \Rightarrow \;\left(\lambda_g = 0 \;\text{for all} \; g\in G \right).
\end{equation*}

The vector space $\CC[G]$ is equipped with a product,\index{product}
often called the \emph{convolution product}\index{convolution product},
defined by the formula
\begin{equation*}
\left(\sum_{g\in G}\, \lambda_g \, g \right) \left(\sum_{g\in G}\, \mu_g \, g \right)
= \sum_{g\in G} \left(\sum_{h\in G}\, \lambda_h \, \mu_{h^{-1}g} \right) g .
\end{equation*}
The convolution product possesses a unit\index{unit}, which is $1_{\CC[G]} = e$, where $e$ is the unit of the group~$G$.
The algebra~$\CC[G]$ is called the \emph{convolution algebra}\index{convolution algebra} of~$G$, 
or simply the \emph{group algebra}\index{group algebra}\index{algebra!group} of~$G$.

We now claim that $\CC[G]$ is a Hopf algebra.  Its coproduct, counit, and antipode are given by
\begin{equation}\label{def-CGcoprod}
\Delta\left(\sum_{g\in G}\, \lambda_g \, g \right) = \sum_{g\in G}\, \lambda_g \, g\otimes g,
\qquad
\eps\left(\sum_{g\in G}\, \lambda_g \, g \right) = \sum_{g\in G}\, \lambda_g ,
\end{equation}
\begin{equation}\label{def-CGanti}
S\left(\sum_{g\in G}\, \lambda_g \, g \right) = \sum_{g\in G}\, \lambda_g \, g^{-1}
= \sum_{g\in G}\, \lambda_{g^{-1}} \, g.
\end{equation}

We can see on Formula\,\eqref{def-CGcoprod} for the coproduct that $\Delta^{\op} = \Delta$, 
which means that the Hopf algebra~$\CC[G]$ is \emph{cocommutative}\index{Hopf algebra!cocommutative}.
By Proposition\,\ref{prop-anti}\,(b) this implies that the antipode~$S$ is an involution, 
which can easily be seen on\,\eqref{def-CGanti}.

\begin{exercise}
Prove the claims in Sect.\,\ref{sssec-CG}.
\end{exercise}

\begin{exercise}\label{exo-duality}
\emph{(Duality between the function algebra and the group algebra)}\index{duality}
Let~$G$ be a \emph{finite} group. Define a bilinear form $\OO(G) \times \CC[G] \to \CC$
by 
\begin{equation*}
\left\langle u, \sum_{g\in G}\, \lambda_g \, g \right\rangle = \sum_{g\in G}\, \lambda_g \, u(g)
\end{equation*}
for all $u\in \OO(G)$, $g\in G$, and $\lambda_g \in \CC$.
It induces a linear map $\omega: \OO(G) \to \CC[G]\, \check{}\,$ by  $\omega(u) = \langle u, - \rangle$ ($u\in \OO(G)$).
Recall that $\CC[G]\, \check{}\,$ is the dual Hopf algebra of~$\CC[G]$, as defined in Exercise\,\ref{exo-dualHopf}.
Prove the following:
\begin{enumerate}
\item[(i)] 
The linear map $\omega: \OO(G) \to \CC[G]\, \check{}\, $ is bijective.

\item[(ii)] 
For all $u, v\in \OO(G)$, $g,h \in G$ we have
\begin{eqnarray*}
\langle uv, g \rangle & = & \langle u,g \rangle \, \langle v,g \,\rangle, \\
\langle \Delta(u) , g\otimes h \rangle & = & \langle u , gh\rangle, \\
\eps(u) & = & \langle u , e\rangle, \\
\langle S(u) , g\rangle & = & \langle u , g^{-1}\rangle .
\end{eqnarray*}

\item[(iii)] 
Deduce that $\omega: \OO(G) \to \CC[G]\, \check{}\,$ is an isomorphism of Hopf algebras.
\end{enumerate}
\end{exercise}

\begin{exercise}\label{exo-duality-abelian}
\emph{(Duality for finite abelian groups)}\index{duality}
Let $G$ be a finite \emph{abelian} group and $\widehat{G} = \Hom(G,\CC^{\times})$ be its group of characters. 
We recall that a \emph{character}\index{character of a group} of~$G$ is a group homomorphism from $G$ 
to the multiplicative group~$\CC^{\times}$ of non-zero complex numbers.
Since any element of~$G$ is of finite order, the values of a character of~$G$ are
roots of unity, which are complex numbers of modulus~$1$.

The set~$\widehat{G}$ is a group under pointwise multiplication; 
it is also called the \emph{Pontryagin dual}\index{Pontryagin dual} of~$G$.

(i) Show that $\widehat{G_1 \times G_2} \cong \widehat{G_1} \times \widehat{G_2}$ whenever 
$G_1$ and $G_2$ are finite abelian groups.

(ii) Determine all characters of a cyclic group of order~$n$ and conclude that 
there is a (non-unique) isomorphism $\widehat{\ZZZ/n} \cong \ZZZ/n$.

(iii) Deduce from\,(i) and\,(ii) that $\widehat{G} \cong G$ for any finite abelian group~$G$.
\end{exercise}

\begin{exercise}
\emph{(The Hopf algebras $\CC[G]$ and $\OO(\widehat{G})$)}
Let $G$ be a finite abelian group and $\widehat{G}$ be its group of characters, as defined in the previous exercise. 
Consider the function algebra~$\OO(\widehat{G})$, which is a Hopf algebra by Sect.\,\ref{sssec-OG}. 
Observe that this Hopf algebra is not only commutative, but also cocommutative since $\widehat{G}$ is abelian 
(see Formula\,\eqref{eq-defOG} for the coproduct).
On the other hand we have the cocommutative Hopf algebra~$\CC[G]$, 
which is commutative because $G$ is abelian.
Prove that the linear map $\CC[G] \to \OO(\widehat{G})$ 
defined by $g\in G \mapsto (\chi \mapsto \chi(g))_{\chi \in \widehat{G}}$
is an isomorphism of Hopf algebras.
\end{exercise}

\begin{exercise}\label{exo-GrH}
\emph{(Group-like elements)}\index{group-like element}

(a) Show that the only group-like elements of a group algebra~$\CC[G]$ are of the form
$\sum_{g\in G}\, \lambda_g \, g$, where all coefficients~$\lambda_g$ are zero, except for exactly one,
which is equal to~$1$.
Deduce a group isomorphism $\Gr(\CC[G]) \cong G$.

(b) Given a finite group~$G$, show that an element $u\in \OO(G)$ is group-like if and only if $u(e) = 1$
and $u(gh) = u(g) u(h)$ for all $g,h \in G$, i.e.\ if and only if $u$ is a character of~$G$.
Deduce a group isomorphism $\Gr(\OO(G)) \cong \widehat{G} = \Hom(G,\CC^{\times})$.
\end{exercise}

\subsection{The Heyneman--Sweedler sigma notation}\label{ssec-Sweedler-notation} 
\index{Heyneman--Sweedler notation}

Let $H$ be a Hopf algebra with coproduct~$\Delta$, counit~$\eps$ and antipode~$S$.
It is often convenient to use the following notation (due to Heyne\-man and Sweedler) for the image of
an element $x\in H$ under the coproduct\index{coproduct}:
\begin{equation*}
\Delta(x) = \sum_{(x)} \, x_{(1)} \otimes x_{(2)}.
\end{equation*}

The coassociativity identity\,\eqref{eq-coass}\index{coassociativity} expressed in this notation becomes
\begin{equation*}\label{HS-coprod}
\sum_{(x)} \, (x_{(1)})_{(1)} \otimes (x_{(1)})_{(2)} \otimes x_{(2)}
= \sum_{(x)} \, x_{(1)} \otimes (x_{(2)})_{(1)} \otimes (x_{(2)})_{(2)} .
\end{equation*}
To simplify we will express both sides of the previous equality by
\begin{equation*}\label{HS-coprod2}
\sum_{(x)} \, x_{(1)} \otimes x_{(2)} \otimes x_{(3)}.
\end{equation*}

In this notation the counitality identity\,\eqref{eq-coun}\index{counitality}\index{counit} becomes
\begin{equation}\label{HS-coun}
\sum_{(x)} \, \eps(x_{(1)}) \, x_{(2)} = x = \sum_{(x)} \, x_{(1)} \, \eps(x_{(2)}) .
\end{equation}
The defining equation\,\eqref{eq-anti} for the antipode\index{antipode} becomes
\begin{equation}\label{HS-anti}
\sum_{(x)} \, S(x_{(1)}) \, x_{(2)} = \eps(x)1 = \sum_{(x)} \, x_{(1)} \, S(x_{(2)}).
\end{equation}

The fact that $\Delta$ is a morphism of algebras can be expressed in this notation by
\begin{equation*}
\sum_{(xy)} \, (xy)_{(1)} \otimes (xy)_{(2)}
= \left( \sum_{(x)} \, x_{(1)} \otimes x_{(2)} \right) \left( \sum_{(y)} \, y_{(1)} \otimes y_{(2)} \right).
\end{equation*}
It is convenient to write the previous right-hand side simply as
\begin{equation*}
\sum_{(x) (y)} \, x_{(1)} y_{(1)} \otimes x_{(2)} y_{(2)}.
\end{equation*}

\section{Quantum groups associated with $SL_2(\CC)$}\label{sec-SL2} 

In this section we will present two Hopf algebras which were discovered in the 1980's and
are quantizations\index{quantization} of the special linear group~$SL_2(\CC)$
and of its Lie algebra~$\gs\gl(2)$, the latter consisting of all $2\times 2$-matrices of trace~$0$.
These Hopf algebras depend on a parameter~$q$. They have the particularity of being
neither commutative, nor cocommutative. They are instances of so-called \emph{quantum groups}\index{quantum group}. 

The term ``quantum group'' was introduced by Drinfeld in his Berkeley 1986 ICM ad\-dress~\cite{Dr2}\footnote{Drinfeld
along with other invited mathematicians from the Soviet Union was prevented by the Soviet authorities 
to attend the conference; in Drinfeld's absence his contribution was read by Cartier.}.
As we mentioned in the introduction, 
the discovery of quantum groups was a major event with spectacular applications in representation theory,
low-dimensional topology and theoretical physics.
The reader may learn more on quantum groups in the monographies \cite{CP, Ja, Ka-QG, KRT, Lu}.

\subsection{The quantum coordinate algebra of~$SL_2(\CC)$}\label{ssec-SL2} 

In Sect.\,\ref{ssec-alggr} we considered the special linear group\index{special linear group}~$SL_N(\CC)$
and its coordinate algebra\index{coordinate algebra}
\begin{equation*}
\OO(SL_N(\CC)) = \CC[(a_{i,j})_{1 \leq i, j\leq N}] / (\det(a_{i,j}) -1). 
\end{equation*}

Let us now restrict to the case $N=2$. For simplicity, set $\SL(2) = \OO(SL_2(\CC))$\index{$\SL(2)$}.
We have
\begin{equation*}
\SL(2) = \CC[a,b,c,d] / (ad-bc -1),
\end{equation*}
where $a = a_{1,1}$, $b = a_{1,2}$, $c = a_{2,1}$ and $d = a_{2,2}$.
We can rewrite Formulas\,\eqref{eq-coprod-GL}--\eqref{eq-anti-GL} for the coproduct~$\Delta$,
the counit~$\eps$ and the antipode~$S$ of the Hopf algebra~$\SL(2)$ in the following compact matrix form:
\begin{equation}\label{eq-coprod-SL}
\Delta
\begin{pmatrix}
a & b \\
c& d
\end{pmatrix}
= 
\begin{pmatrix}
a & b \\
c& d
\end{pmatrix}
\otimes \begin{pmatrix}
a & b \\
c & d
\end{pmatrix},
\end{equation}
\begin{equation}\label{eq-counit-SL}
\eps 
\begin{pmatrix}
a & b \\
c& d
\end{pmatrix} 
=
\begin{pmatrix}
1 & 0 \\
0 & 1
\end{pmatrix},
\end{equation}
\begin{equation}\label{eq-anti-SL}
S
\begin{pmatrix}
a & b \\
c& d
\end{pmatrix} 
=
\begin{pmatrix}
d & -b \\
-c & a
\end{pmatrix}.
\end{equation}
This is a compact version for the formulas
\begin{eqnarray*}
\Delta(a) = a \otimes a + b \otimes c, && \Delta(b) = a \otimes b + b \otimes d,\\
\Delta(c) =  c \otimes a + d \otimes c, && \Delta(d) =  c \otimes b + d \otimes d,
\end{eqnarray*}
\begin{equation*}
\eps(a) = \eps(d) = 1, \quad \eps(b) = \eps(d) = 0,
\end{equation*}
\begin{equation*}
S(a) = d, \quad S(b) = -b, \quad S(c) = -c, \quad S(d) = a.
\end{equation*}

The Hopf algebra~$\SL(2)$ is commutative, but not cocommutative, which can be seen for instance on 
the formula for $\Delta(a)$. Its antipode is clearly an involution, which follows of course from the fact that 
the map $\inv:g \mapsto g^{-1}$ is involutive.

Now we introduce a \emph{non-commutative deformation}\index{$q$-deformation}\index{deformation} 
of the Hopf algebra~$\SL(2)$. 
The deformation depends on a parameter~$q$ which we take to be a non-zero complex number.
Define $\SL_q(2)$\index{quantum $SL(2)$}\index{$\SL_q(2)$} to be 
the algebra generated by four generators $a$, $b$, $c$, $d$ subject to the relations
\begin{eqnarray*}
ba = qab, && ca = qac, \\
db = qbd, && dc = qcd, \\
bc = cb, && ad - da = (q^{-1} - q) \, bc,\\
ad - q^{-1} bc &=& 1.
\end{eqnarray*}

If $q=1$, the previous relations reduce to the fact that the generators $a$, $b$, $c$, $d$ commute and satisfy
the additional relation $ad-bc = 1$. Thus in this case, we have $\SL_q(2) = \SL(2)$.
If $q \neq 1$, then clearly $\SL_q(2)$ is not commutative, so it cannot be isomorphic to~$\SL(2)$.

The algebra $\SL_q(2)$ is a Hopf algebra\index{Hopf algebra}. Its coproduct~$\Delta$ and counit~$\eps$
are given by the same formulas as for~$\SL(2)$, namely by\,\eqref{eq-coprod-SL} and\,\eqref{eq-counit-SL}.
However the antipode~$S$ of~$\SL_q(2)$ is given, not by\,\eqref{eq-anti-SL},
but by another formula (depending on~$q$), namely in compact matrix form by
\begin{equation}\label{eq-anti-SLq}
S
\begin{pmatrix}
a & b \\
c& d
\end{pmatrix} 
=
\begin{pmatrix}
d & -qb \\
-q^{-1}c & a
\end{pmatrix}.
\end{equation}

The Hopf algebra\index{Hopf algebra!non-commutative non-cocommutative} $\SL_q(2)$ 
provides our first example of a Hopf algebra that is (for general~$q$) \emph{neither commutative}, 
\emph{nor cocommutative}, and with non-involutive antipode (for the latter, see Exercise\,\ref{antipode-order} below). 
The Hopf algebra $\SL_q(2)$ is a \emph{quantization}\index{quantization} of the coordinate algebra~$\SL(2)$;
this is another way of saying that $\SL_q(2)$ is a deformation of~$\SL(2)$ as a Hopf algebra.

The Hopf algebra $\SL_q(2)$ is an example of a \emph{quantum group}\index{quantum group}. 
The Hopf algebras $\OO(GL_N(\CC))$ and $\OO(SL_N(\CC))$ can be quantized in a similar fashion.

\begin{exercise}
(a) Compute the following expressions in~$\SL_q(2) \otimes \SL_q(2)$ involving 
the coproduct~$\Delta$ defined by\,\eqref{eq-coprod-SL}:
\begin{eqnarray*}
\Delta(b)\Delta(a) - q\Delta(a)\Delta(b), && \Delta(c)\Delta(a) - q\Delta(a)\Delta(c), \\
\Delta(d)\Delta(b) - q\Delta(b)\Delta(d), && \Delta(d)\Delta(c) - q\Delta(c)\Delta(d), \\
\Delta(b)\Delta(c) - \Delta(c)\Delta(b), && \Delta(a)\Delta(d) - q^{-1} \Delta(b)\Delta(c) - 1\otimes 1, \\
\Delta(a)\Delta(d) - \Delta(d)\Delta(a) & - & (q-q^{-1}) \, \Delta(b)\Delta(c).
\end{eqnarray*}
Deduce that $\Delta: \SL_q(2) \to \SL_q(2) \otimes \SL_q(2)$ is a morphism of algebras.

(b) Check that $\SL_q(2)$ satisfies all axioms of a Hopf algebra.
\end{exercise}

\begin{exercise}\label{antipode-order}
\emph{(The square of the antipode)}

(a) Use \eqref{eq-anti-SLq} to compute the square~$S^2$ of the antipode of~$\SL_q(2)$ on the generators $a$, $b$, $c$,~$d$.

(b) Show that~$S^2$ has infinite order if $q$ is not a root of unity\index{root of unity}.

(c) If $q = \exp(\pi\sqrt{-1}/N)$ for some integer $N > 1$, prove that $S^2$ is a Hopf algebra 
automorphism of~$\SL_q(2)$ of order~$N$. 
\end{exercise}

\begin{exercise}
For $\eps = \pm 1$ define $\SL_{(\eps)}(2)$ to be the algebra generated by $X,Y,Z,T$ and the relations
$XY = YX$, $XZ = ZX$, $XT = TX$, $YZ  = \eps ZY$, $YT = \eps TY$, $ZT = \eps TZ$ and
$X^2 - \eps Y^2 -\eps Z^2 + \eps T^2 = 1$.

(a) Let $\eps = 1$. Show that 
there is an algebra isomorphism $\varphi: \SL_{(\eps)}(2) \to \SL(2)$ such that $\varphi(X) = (a+d)/2$,
$\varphi(Y) = (a-d)/2$, $\varphi(Z) = (b+c)/2$, $\varphi(T) = (b-c)/2$.
Deduce $\Alg(\SL_{(\eps)}(2), \CC) \cong SL_2(\CC)$.

(b) Let $\eps = -1$. Show that $\Alg(\SL_{(\eps)}(2), \CC)$ is the union of three quadrics
lying in three distinct planes (for further details, see\,\cite[Sect.\,4.2]{GKM}).
\end{exercise}

\subsection{A quotient of~$\SL_q(2)$}\label{ssec-noco-nococo}
\index{Hopf algebra!non-commutative non-cocommutative}

Let $q$ be again a non-zero scalar. 
Consider the algebra $\CC_q[X,X^{-1},Y]$\index{$\CC_q[X,X^{-1},Y]$} generated by three generators $X, X^{-1},Y$ 
subject to the relations 
\begin{equation*}
X X^{-1} = X^{-1} X = 1, \qquad YX = qXY.
\end{equation*}
This algebra is non-commutative when $q\neq 1$.
Proceeding as in Exercise\,\ref{exo-qplane-basis}, 
the reader may check that the set $\{X^i Y^j\}$ where $i$~runs over~$\ZZZ$ and $j$ over~$\NN$
is a basis of~$\CC_q[X,X^{-1},Y]$. 
The algebra~$\CC_q[X,X^{-1},Y]$ contains the quantum plane\index{quantum plane}~$\CC_q[X,Y]$ 
of~Sect.\,\ref{sssec-qplane} as a subalgebra.

The algebra~$\CC_q[X,X^{-1},Y]$ has the structure of a Hopf algebra\index{Hopf algebra} with coproduct~$\Delta$,
counit~$\eps$ and antipode~$S$ given on the generators $X,Y$ by
\begin{equation}\label{eq-copr-qLaurent}
\Delta(X) = X \otimes X , \qquad \Delta(Y) = X \otimes Y + Y \otimes X^{-1},
\end{equation}
\begin{equation}\label{eq-counit-qLaurent}
\eps(X) = 1, \quad \eps(Y) = 0, \quad S(X) = X^{-1} , \quad S(Y) = -qY.
\end{equation}
The formula for $\Delta(Y)$ shows that $\CC_q[X,X^{-1},Y]$ is a non-cocom\-mu\-ta\-tive Hopf algebra.

Moreover, $\CC_q[X,X^{-1},Y]$ is a quotient of the Hopf algebra~$\SL_q(2)$\index{quantum $SL(2)$} 
introduced in Sect.\,\ref{ssec-SL2}; we have the following precise statement, whose proof we leave to the reader.

\begin{lemma}\label{lem-noco-nococo}
There is a surjective morphism of Hopf algebras
\[
\pi : \SL_q(2) \to \CC_q[X,X^{-1},Y]
\]
such that
$\pi(a) = X$, $\pi(b) = Y$, $\pi(c) = 0$, and $\pi(d) = X^{-1}$.
\end{lemma}

Since the morphism~$\pi$ kills the generator~$c$ of~$\SL_q(2)$, we can see~$\CC_q[X,X^{-1},Y]$ 
as a quantization\index{quantization} of the coordinate algebra 
of the subgroup~$B$ of upper triangular matrices in~$SL_2(\CC)$.

\subsection{The quantum enveloping algebra of~$\gs\gl(2)$}\label{ssec-qea}

We now describe another important quantum group\index{quantum group}, 
which is dual to the quantum group~$\SL_q(2)$
in a sense which will be made precise in Lemma\,\ref{lem-dualSL2} below.
 
This new algebra, denoted~$U_q\, \gs\gl(2)$\index{$U_q\, \gs\gl(2)$}, also depends on a non-zero complex parameter~$q$;
we furthermore assume $q \neq \pm 1$, so that $q - q^{-1} \neq 0$.

We define $U_q\, \gs\gl(2)$ to be the algebra generated by four elements $E,F, K, K^{-1}$
subject to the relations
\begin{equation*}
K K^{-1} = K^{-1} K = 1,
\end{equation*}
\begin{equation*}
KE = q^2 EK ,
\qquad
KF = q^{-2} FK ,
\end{equation*}
\begin{equation*}
EF - FE = \frac{K - K^{-1}}{q - q^{-1}} .
\end{equation*}

The algebra $U_q\, \gs\gl(2)$ is called the \emph{quantum enveloping algebra}\footnote{The
concept of enveloping algebra of a Lie algebra is a classical concept of the theory of Lie algebras;
see for instance \cite{Di, Jc, Ka-QG, Se}. The relationship between the 
quantum enveloping algebra~$U_q\, \gs\gl(2)$ and the enveloping algebra
of the Lie algebra~$\gs\gl(2)$ is explained in \cite[VI.2]{Ka-QG}.}
\index{quantum enveloping algebra} of
the Lie algebra~$\gs\gl(2)$. 
The set $\{E^i F^j K^{\ell}\}_{i,j \in \NN; \, \ell\in \ZZZ}$ is a basis of~$U_q\, \gs\gl(2)$ 
considered as a complex vector space (for a proof, see\,\cite[Prop.\,VI.1.4]{Ka-QG}).

The algebra $U_q\, \gs\gl(2)$ is a Hopf algebra with coproduct\index{coproduct}~$\Delta$, counit\index{counit}~$\eps$,
and antipode\index{antipode}~$S$ given on the generators by
\begin{equation*}
\Delta(K^{\pm 1}) = K^{\pm 1} \otimes K^{\pm 1}, \quad \eps(K^{\pm 1}) = 1, \quad S(K^{\pm 1}) = K^{\mp 1},
\end{equation*}
\begin{equation*}
\Delta(E) = 1 \otimes E + E \otimes K, \quad \eps(E) = 0, \quad S(E) = -EK^{-1},
\end{equation*}
\begin{equation*}
\Delta(F) = K^{-1} \otimes F + F \otimes 1, \quad \eps(F) = 0, \quad S(F) = -q^{-1}FK.
\end{equation*}

The algebra $U_q\, \gs\gl(2)$ first appeared in a paper by Kulish and Reshetikhin; 
its Hopf algebra structure is due to Sklyanin (\emph{cf.}\,\cite{KR, Sk}).

Consider the morphism of algebras $\rho : U_q\, \gs\gl(2) \to M_2(\CC)$ given by
\begin{equation*}
\rho(K^{\pm 1}) = 
\begin{pmatrix}
q^{\pm 1} & 0 \\
0 & q^{\mp 1}
\end{pmatrix},
\quad
\rho(E) =
\begin{pmatrix}
0 & 1\\
0 & 0
\end{pmatrix},
\quad
\rho(F) =
\begin{pmatrix}
0 & 0\\
1 & 0
\end{pmatrix}.
\end{equation*}
It is a two-dimensional representation of~$U_q\, \gs\gl(2)$.
For any $u \in U_q\, \gs\gl(2)$, the matrix $\rho(u)$ is of the form
\begin{equation*}
\rho(u) =
\begin{pmatrix}
A(u) & B(u) \\
C(u) & D(u)
\end{pmatrix}.
\end{equation*}
This equality defines four linear forms $A,B,C,D$ on $U_q\, \gs\gl(2)$, hence four elements
$A,B,C,D$ on the dual algebra~$U_q\, \gs\gl(2)\, \check{}\,$ whose product is given by\,\eqref{dual-prod}.

\begin{lemma}\label{lem-dualSL2}\index{$\SL_q(2)$}
There is a unique morphism of algebras $\psi: \SL_q(2) \to U_q\, \gs\gl(2)\, \check{}\,$ such that
\begin{equation*}
\psi(a) = A, \quad \psi(b) = B, \quad \psi(c) = C, \quad \psi(d) = D.
\end{equation*}
\end{lemma}

For a proof we refer to\,\cite[Sect.\,VII.4]{Ka-QG}. Takeuchi\,\cite{Ta1} showed that $\psi$ is injective;
thus $\SL_q(2)$ embeds into the dual of the quantum enveloping algebra~$U_q\, \gs\gl(2)$.
Actually, the image of the morphism $\psi$ lies inside the restricted dual Hopf algebra~$U_q\, \gs\gl(2)^{\circ}$,
as defined in Remark\,\ref{rem-restricted}.

\begin{exercise}
Prove that the map $\rho : U_q\, \gs\gl(2) \to M_2(\CC)$ defined above is a morphism of algebras.
Give a proof of Lemma\,\ref{lem-dualSL2}.
\end{exercise}

\begin{exercise}
Check that the group-like elements of~$U_q\, \gs\gl(2)$ consist of the powers $K^k$ of~$K$
($k\in \ZZZ$).
\end{exercise}

\begin{exercise}
Show that the following element of~$U_q\, \gs\gl(2)$ belongs to its center:
\begin{equation*}
EF + \frac{q^{-1}K + qK^{-1}}{(q - q^{-1})^2} .
\end{equation*}
\end{exercise}

\begin{remark}\label{rem-Uqg}
Drinfeld\,\cite{Dr1, Dr2} and Jimbo\,\cite{Ji} generalized the construction of $U_q\, \gs\gl(2)$ to any symmetrizable
Kac--Moody Lie algebra~$\gog$. The resulting Hopf algebra~$U_q\, \gog$ is a quantization\index{quantization} of
the universal enveloping algebra of~$\gog$.
\end{remark}

\subsection{A finite-dimensional quotient of $U_q\, \gs\gl(2)$}\label{ssec-ud}

The quantum enveloping algebra $U_q\, \gs\gl(2)$ has an interesting quotient when
$q$ is a root of unity\index{root of unity} of order~$d$ ($d \geq 3$ since $q \neq \pm 1$).
Assume $q$ is such a root of unity.
Set $e = d$ if $d$ is odd, and $e = d/2$ if $d$ is even; we have $e \geq 2$.

Let $I$ be the two-sided ideal of~$U_q\, \gs\gl(2)$ generated by $E^e$, $F^e$ and $K^e - 1$.
Define the quotient algebra\index{$\gu_d$}
\[
\gu_d = U_q\, \gs\gl(2)/I.
\] 
It can be shown that the set $\{E^i F^j K^{\ell}\}_{1 \leq i,j,\ell \leq e-1}$ of elements of~$U_q\, \gs\gl(2)$
maps to a basis of~$\gu_d$ (for a proof, see\,\cite[Prop.\,VI.5.8]{Ka-QG}). 
Therefore, $\gu_d$ is finite-dimensional of dimension equal to~$e^3$.

Moreover, there is a unique Hopf algebra structure on~$\gu_d$ such that 
the natural projection $U_q\, \gs\gl(2) \to \gu_d$ is a morphism of Hopf algebras (see\,\cite[Prop.\,IX.6.1]{Ka-QG}).

\begin{exercise}
Let $q$ be a root of unity of order $d\geq 3$ and $e$ as above.
Show that the elements $E^e$, $F^e$, $K^e$ lie in the center of~$U_q\, \gs\gl(2)$.
\end{exercise}
 
We will come back to~$U_q\, \gs\gl(2)$ and~$\gu_d$ in Sect.\,\ref{sec-deform-Uq}.

\section{Group actions in non-commutative geometry}\label{sec-group-action}

Our next step is to extend the concept of a group action to the world of non-commutative spaces. 
We need to introduce the concept of a comodule algebra over a Hopf algebra.
As we shall see, such a concept covers various situations.

\subsection{Comodule-algebras}\label{ssec-comod}

Fix a Hopf algebra~$H$ with coproduct $\Delta$ and counit~$\eps$.

\begin{definition}\label{def-comod}
A \emph{(right) $H$-comodule algebra}\index{comodule algebra}\index{algebra!comodule} 
is an (associative unital) algebra~$A$ equipped with a morphism of algebras $\delta = A \to A \otimes H$, 
called the \emph{coaction}\index{coaction}, satisfying the following properties:

(a) {(Coassociativity)}\index{coassociativity}
\begin{equation}\label{coass-comod}
(\delta \otimes \id_H) \circ \delta = (\id_A \otimes \,\Delta) \circ \delta,
\end{equation}

(b) {(Counitarity)}\index{counitarity}
\begin{equation}\label{counit-comod}
(\id_A \otimes \, \eps)\circ \delta = \id_A,
\end{equation}
where we have identified $A\otimes \CC$ with~$A$.
\end{definition}

Any $H$-comodule algebra~$A$ contains a subalgebra, which will turn out to be of importance to us,
namely the subalgebra of~$A$ on which the coaction $\delta$ is trivial:
\begin{equation*}
A^{\co-H} = \left\{ a \in A \; | \; \delta(a) = a \otimes 1 \right\}.
\end{equation*}
The elements of~$A^{\co-H}$\index{$A^{\co-H}$} are called \emph{coinvariant}\index{coinvariant}.

\begin{exercise}
Show that $A^{\co-H}$ is a subalgebra of~$A$ and that the unit~$1_A$ of~$A$ belongs to~$A^{\co-H}$.
\end{exercise}

The following example of a comodule algebra shows that this concept extends group actions to non-commutative algebra.

\begin{example}\label{exa-OGcoation}
Let $G$ be a finite group acting on the right on a finite set~$X$. Then the action, which is a map
$X \times G \to X$ induces a morphism of algebras~$\delta$ between the corresponding function algebras\index{function algebra}
\[
\delta: \OO(X) \to \OO(X \times G) = \OO(X) \otimes \OO(G).
\]
Equipped with~$\delta$, the algebra~$\OO(X)$ becomes an $H$-comodule algebra 
for the Hopf algebra $H= \OO(G)$.

Let $Y= X/G$ be the set of orbits\index{orbit space} of the action of~$G$ on~$X$. Then the projection $X \to Y$
sending each element~$x \in X$ to its orbit~$xG$ induces an injective morphism of algebras $\OO(Y) \to \OO(X)$. 
It can be checked that $\OO(Y)$ coincides with the subalgebra~$\OO(X)^{\co-\OO(G)}$\index{$A^{\co-H}$}
of coinvariant elements of~$\OO(X)$.
\end{example}

\begin{example}\label{exa-H=comodalg}
In Definition\,\ref{def-comod} set $A$ to be equal to the Hopf algebra~$H$ and the coaction~$\delta$
to be equal to the coproduct~$\Delta$ of~$H$.
Then $H$ becomes an $H$-comodule algebra.
We claim that any coinvariant\index{coinvariant} element $x\in H$ is a scalar multiple of the unit~$1$ of~$H$. 
Indeed, applying $\eps \otimes \id$ to both sides of the equality $\Delta(x) = x \otimes 1$ and using\,\eqref{eq-coun},
we obtain $x = \eps(x) \, 1$, which yields the desired conclusion.
\end{example}

We now give more examples of comodule algebras.

\subsection{Group-graded algebras}\label{ssec-super}

Let $G$ be a group. 

\begin{definition}\label{def-graded}
A \emph{$G$-graded algebra}\index{graded algebra}\index{algebra!graded} 
is an algebra~$A$ together with a vector space decomposition
\[
A = \bigoplus_{g\in G} \, A_g,
\]
where each $A_g$ is a linear subspace of~$A$ such that

(a) $A_g A_h \subset A_{gh}$ for all $g,h \in G$, which means that the product $a b$ belongs to~$A_{gh}$
whenever $a \in A_g$ and $b\in A_h$;

(b) the unit\index{unit}~$1_A$ of~$A$ is in~$A_e$, where $e$ is the unit of the group~$G$.
\end{definition}

It follows from the definition that $A_e$ is a subalgebra of~$A$ and that each $A_g$ is an $A_e$-bimodule
under the product of~$A$.

When $G = \ZZZ/2$ is the cyclic group of order~$2$, then a $G$-graded algebra is often called
a \emph{superalgebra}\index{superalgebra}.

We next show that a $G$-graded algebra is the same as a $\CC[G]$-comodule algebra,
where $\CC[G]$ is the convolution algebra\index{group algebra}\index{algebra!group} of the group~$G$
with its Hopf algebra structure defined in Sect.\,\ref{sssec-CG} (see also~\cite[Lemma\,4.8]{BlM}).

\begin{proposition}\label{prop-graded}
(a) Any $G$-graded algebra~$A$ is a $\CC[G]$-comodule algebra. 
Moreover, $A^{\co-\CC[G]} = A_e$\index{$A^{\co-H}$}.

(b) Conversely, any $\CC[G]$-comodule algebra is a $G$-graded algebra. 
\end{proposition}

\begin{proof}
(a) We define a linear map $\delta: A \to A \otimes \CC[G]$ by
\begin{equation*}
\delta(a) = a \otimes g \quad \text{for all} \; a \in A_g.
\end{equation*}
The map $\delta$ is a morphism of algebras in view of Conditions\,(a) and\,(b) of Definition\,\ref{def-graded}.
Let us check the coassociativity and counitarity conditions of Definition\,\ref{def-comod} for~$\delta$.
Firstly, for any $a\in A_g$,
\begin{equation*}
(\delta \otimes \id_H) \circ \delta (a) = (\delta \otimes \id_H)(a \otimes g) = a\otimes g \otimes g .
\end{equation*}
Similarly, 
\begin{equation*}
(\id_A \otimes \, \Delta) \circ \delta (a) = (\id_A \otimes \, \Delta)(a \otimes g) = a\otimes g \otimes g
\end{equation*}
in view of\,\eqref{def-CGcoprod}. 
Therefore, $(\delta \otimes \, \id_H) \circ \delta = (\id_A \otimes \, \Delta) \circ \delta$
holds on each subspace~$A_g$, hence on~$A$.
Secondly, for any $a\in A_g$,
\begin{equation*}
(\id_A \otimes \, \eps)\circ \delta (a) = (\id_A \otimes \, \eps) (a \otimes g) = a \, \eps(g) = a
\end{equation*}
again in view of\,\eqref{def-CGcoprod}. 

The inclusion $A_e \subset A^{\co-\CC[G]}$ follows from the definition of~$\delta$
and from the fact that $e$ is the unit of~$\CC[G]$.
Let us prove the converse inclusion.
For a general element $a = \sum_{g\in G} \, a_g \in A$ with each $a_g \in A_g$, we have
\begin{equation*}
\delta(a) = \sum_{g\in G} \, a_g \otimes g .
\end{equation*}
Since the elements $g\in G$ are linearly independent in~$\CC[G]$, we see that, if $a$ is coinvariant, i.e., 
$\delta(a) = a \otimes e$, then $a_g = 0$ for all $g\neq e$. Thus any coinvariant element belongs to~$A_e$.

(b) Assume now that $A$ is a $\CC[G]$-comodule algebra with coaction~$\delta$.
Using the natural basis $\{g\}_{g\in G}$ of~$\CC[G]$, we can expand $\delta(a) \in A \otimes \CC[G]$ 
for any $a\in A$ uniquely as
\[
\delta(a) = \sum_{g\in G} \, p_g(a) \otimes g
\]
where each $p_g(a)$ belongs to~$A$. 
It is clear that $a\mapsto p_g(a)$ defines a linear endomorphism~$p_g$ of~$A$.

Let us now express the coassociativity of the coaction~$\delta$. On one hand, we have
\[
(\delta \otimes \id_H) \circ \delta (a) = (\delta \otimes \id_H) \left(\sum_{g\in G} \, p_g(a) \otimes g \right)
= \sum_{g\in G} \, \sum_{h\in G} \,p_h(p_g(a)) \otimes h\otimes g.
\]
On the other hand, 
\[
(\id_A \otimes \, \Delta) \circ \delta (a) = (\id_A \otimes \, \Delta)\left(\sum_{g\in G} \, p_g(a) \otimes g \right)
= \sum_{g\in G} \, p_g(a) \otimes g\otimes g.
\]
Identifying both right-hand sides in view of \eqref{coass-comod}, we obtain
\begin{equation}\label{eq-idempotent}
p_h \circ p_g = 
\begin{cases}
\, p_g & \text{if} \; g = h,\\
\; 0 & \text{otherwise.}
\end{cases}
\end{equation}

Next, the counitarity condition~\eqref{counit-comod} implies that
\begin{eqnarray*}
a & = & (\id_A \otimes \, \eps)\circ \delta(a) = (\id_A \otimes \, \eps)\left(\sum_{g\in G} \, p_g(a) \otimes g \right) \\
& = & \sum_{g\in G} \, p_g(a) \, \eps(g) = \sum_{g\in G} \, p_g(a) .
\end{eqnarray*}
In other words, 
\begin{equation}\label{eq-sum}
\sum_{g\in G} \, p_g = \id_A .
\end{equation}

Define the linear subspace $A_g = p_g(A)$ of~$A$ for all $g\in G$. 
The equality\,\eqref{eq-sum} implies $\sum_{g\in G} \, A_g = A$.
Let us check that this sum is a direct sum. Indeed, let us assume that
$\sum_{g\in G} \, p_g(a_g) = 0$ in~$A$ for a family~$(a_g)$ of elements of~$A$
and apply $p_h$ to it for a fixed element $h\in G$.
By\,\eqref{eq-idempotent}, we obtain
\[
0  = p_h \left( \sum_{g\in G} \, p_g(a_g) \right) = \sum_{g\in G} \, p_h(p_g(a_g) ) = p_h(a_h).
\]
Since this holds for any $h\in G$, we see that each summand in the sum $\sum_{g\in G} \, p_g(a_g)$ vanishes.

We claim that $\delta(a) = a \otimes g$ for any $a\in A_g$. Indeed, an element of~$A_g$ is of the form
$a = p_g(a')$ for some~$a' \in A$. Using\,\eqref{eq-idempotent},
we obtain
\[
\delta(a) = \sum_{h\in G} \, p_h(a) \otimes h = \sum_{h\in G} \, p_h(p_g(a')) \otimes h = p_g(a') \otimes g
= a \otimes g.
\]

It remains to check that $a b$ belongs to~$A_{gh}$ for all $a \in A_g$ and $b\in A_h$, and that $1_A$ belongs to~$A_e$.
For the first requirement, we have
$\delta(a) = a\otimes g$ and $\delta(b) = b\otimes h$. Since $\delta$ is a morphism of algebras,
we have
\[
\delta(ab) = \delta(a) \delta(b) = (a\otimes g) (b \otimes h) = ab \otimes gh,
\]
which proves that the product $ab$ belongs to~$A_{gh}$. 

For the second requirement, we have $\delta(1_A) = 1_A \otimes \, e$; 
thus, the unit of the algebra belongs to the component~$A_e$ indexed by the unit~$e$ of the group.
\qed
\end{proof}

Let us give a few examples of group-graded algebras.\index{graded algebra}\index{algebra!graded} 

\begin{example}
By Example\,\ref{exa-H=comodalg} we know that the Hopf algebra~$\CC[G]$ is itself a $\CC[G]$-comodule algebra 
with coaction equal to the coproduct~$\Delta$ of~$\CC[G]$. 
Since $\Delta(g) = g \otimes g$ by\,\eqref{def-CGcoprod}, we deduce from Proposition\,\ref{prop-graded} and its proof that
$\CC[G]$ is a $G$-graded algebra $\CC[G] = \bigoplus_{g\in G}\, A_g$, where each $g$-component~$A_g$
is one-dimensional and consists of all scalar multiples of the element~$g$.
\end{example}

\begin{example}\label{exa-graded-c}
\emph{(Gradings on matrix algebras)}

(a) Consider the algebra\index{matrix algebra}\index{algebra!matrix}~$M_N(\CC)$ of $N \times N$-matrices.
Let $E_{i,j} \in M_N(\CC)$ be the matrix whose entries are all zero, except for the $(i,j)$-entry which is equal to~$1$.
The $N^2$ matrices~$E_{i,j}$ ($1 \leq i,j \leq N$) form a basis of~$M_N(\CC)$.

The algebra~$M_N(\CC)$ can be given many group gradings. 
Indeed, let $G$ be a group and $(g_1, \ldots, g_N)$ be an $N$-tuple of elements of~$G$.
For any $g \in G$, let $A_g$ be the vector space spanned by all matrices~$E_{i,j}$ such that $g_ig_j^{-1} = g$;
we set $A_g = 0$ is there is no couple $(i,j)$ such that $g_ig_j^{-1} = g$.
Then the decomposition $M_N(\CC) =  \bigoplus_{g\in G}\, A_g$ yields the structure of a $G$-graded algebra on~$M_N(\CC)$
(check this claim!).

(b) As a special case of the previous gradings, take $G = \ZZZ/N$ to be the cyclic group 
generated by an element~$t$ of order~$N$ and 
\[
(g_1, \ldots, g_N) = (e, t, t^2, \dots, t^{N-1}).
\] 
Then $M_N(\CC)$ has a grading $M_N(\CC) =  \bigoplus_{k=0}^{N-1}\, A_{t^k}$ 
for which $A_{t^k}$ consists of all matrices $(a_{i,j})_{1 \leq i,j \leq N}$
such that $a_{i,j} = 0$ if $i-j \not\equiv k \pmod{N}$. In particular, $A_e$ is the subalgebra of diagonal matrices.
Each $A_{t^k}$ is $N$-dimensional.
\end{example}

\begin{example}\label{exa-graded-d}
Let $\HH$ be the four-dimensional algebra of \emph{complex quaternions}.\index{quaternion algebra}
Recall that it has a basis $\{1, i, j, k\}$ such that the
multiplication of~$\HH$ is given by the following rules : $1$ is the unit and
\begin{equation*}
i^2 = j^2 = k^2 = -1, \quad 
ij = -ji = k, \quad jk = -kj = i, \quad ki = - ik = j.
\end{equation*}
The algebra~$\HH$ is $G$-graded, where $G$ is the group $(\ZZZ/2)^2$ of order~$4$: we have
\begin{equation*}
A_{(0,0)} = \CC \, 1, \quad A_{(1,0)} = \CC \, i, \quad A_{(0,1)} = \CC \, j, \quad A_{(1,1)} = \CC \, k.
\end{equation*}
There is an isomorphism of algebras $\psi: \HH \to M_2(\CC)$ given by
\begin{eqnarray*}
\psi(1) =
\begin{pmatrix}
1 & 0 \\
0 & 1 
\end{pmatrix},
&&
\psi(i) =
\begin{pmatrix}
0 & \sqrt{-1} \\
\sqrt{-1} & 0 
\end{pmatrix}, 
\\
\psi(j) =
\begin{pmatrix}
0 & -1 \\
1 & 0 
\end{pmatrix},
&&
\psi(k) =
\begin{pmatrix}
\sqrt{-1} & 0 \\
0 & - \sqrt{-1} 
\end{pmatrix}.
\end{eqnarray*}
This isomorphism induces a $(\ZZZ/2)^2$-grading on~$M_2(\CC)$. 
Such a grading is not of the form presented in Example\,\ref{exa-graded-c}\,(b) above.
\end{example}

\subsection{Algebras with group actions}\label{ssec-Galg}

Let $G$ be a group. 

\begin{definition}\label{def-Galg}
A \emph{$G$-algebra}\index{$G$-algebra}\index{algebra with group action} is an algebra~$A$ 
together with a group homomorphism $\rho: G \to \Aut(A)$
such that each $\rho(g)$ is an algebra automorphism of~$A$.
\end{definition}

The subspace $A^G$ consisting of all elements $a\in A$ such that $\rho(g)(a) = a$ for all $g\in A$
forms a subalgebra of~$G$. The elements of~$A^G$ are called \emph{$G$-invariants}\index{invariant}.

Any algebra has the structure of a $G$-algebra with $G$ taken to be (a subgroup of) the group
of algebra automorphisms of~$A$. Let us give a few more examples of $G$-algebras.

\begin{example}
If $K$ is a finite \emph{Galois extension}\index{Galois extension} of a number field~$k$ with Galois group~$G$, 
then $G$ acts by automorphisms on~$K$ and we have $K^G = k$.
\end{example}

\begin{example}
The general linear group $GL_N(\CC)$ acts by conjugation on the matrix algebra~$M_N(\CC)$.
The $GL_N(\CC)$-invariants are the scalar multiples of the identity matrix.
\end{example}

Assume now that the group~$G$ is finite. Consider the Hopf algebra~$\OO(G)$ 
(introduced in Sect.\,\ref{sssec-OG}) and its basis~$\{\delta_g\}_{g\in G}$ of $\delta$-functions.

\begin{proposition}\label{prop-G-alg}
(a) Any $G$-algebra~$A$ is an $\OO(G)$-comodule algebra\index{comodule algebra}\index{algebra!comodule} 
with coaction $\delta: A \to A \otimes \OO(G)$ given for all $a \in A$ by
\begin{equation*}
\delta(a) = \sum_{g\in G} \, \rho(g)(a) \otimes \delta_g.
\end{equation*}
Moreover, the subalgebra~$A^{\co-\OO(G)}$\index{$A^{\co-H}$} of coinvariant\index{coinvariant} elements coincides
with the subalgebra~$A^G$ of $G$-invariant elements of~$A$:
\[
A^{\co-\OO(G)} = A^G .
\]

(b) Conversely, any $\OO(G)$-comodule algebra is a $G$-algebra. 
\end{proposition}

The proof is left to the reader, who is invited to take inspiration from the proof of Proposition\,\ref{prop-graded}.

\subsection{The quantum plane and its $\SL_q(2)$-coaction}\label{ssec-qplane}

The special linear group $SL_2(\CC)$ 
acts on the two-dimensional vector space~$\CC^2$ by matrix multiplication. 
As a special case of Example\,\ref{exa-OGcoation}, the coordinate algebra~$\CC[X,Y]$
of~$\CC^2$ becomes a $\SL(2)$-comodule algebra\index{$\SL(2)$}. 
Recall from Sect.\,\ref{ssec-SL2} that 
\[
\SL(2) = \CC[a,b,c,d] / (ad-bc-1)
\]
is the coordinate algebra\index{coordinate algebra} of~$SL_2(\CC)$.\index{special linear group}
It is easy to check that the corresponding coaction\index{coaction}
$\delta: \CC[X,Y] \to \CC[X,Y] \otimes SL_2(\CC)$
is given by
\begin{equation}\label{eq-qplane}
\delta(X,Y) = (X,Y) \otimes
\begin{pmatrix}
a & b \\
c & d
\end{pmatrix},
\end{equation}
which is short for
\begin{equation*}
\delta(X) = X \otimes a + Y \otimes c
\quad\text{and} \quad
\delta(Y) = X \otimes b + Y \otimes d .
\end{equation*}

In Sect.\,\ref{ssec-SL2} we quantized $\SL(2)$ using a complex parameter $q \neq 0$.
We now proceed to quantize\index{$q$-deformation}\index{quantization}\index{deformation} the previous coaction.
To this end we replace $\CC[X,Y]$ by the quantum plane\index{quantum plane} 
$\CC_q[X,Y] = \CC \, \langle X,Y \rangle / (YX - qXY)$\index{$\CC_q[X,Y]$} introduced in Sect.\,\ref{sssec-qplane}.

\begin{theorem}\label{th-qplane}\index{$\SL_q(2)$}
The map $\delta$ given by Formula\,\eqref{eq-qplane} equips the quantum plane $\CC_q[X,Y]$
with the structure of a $\SL_q(2)$-comodule algebra\index{comodule algebra}\index{algebra!comodule}. 
Moreover, the subalgebra of coinvariants\index{coinvariant} of~$\CC_q[X,Y]$ is~$\CC 1$.
\end{theorem}

The second assertion is the non-commutative analogue of the fact that the only point of the plane 
which is invariant under the action of~$\SL_2(\CC)$ is the origin.

\begin{proof}
(a) We first have to establish that $\delta$ is a morphism of algebras. It suffices to check that
$\delta(Y) \delta(X) = q \, \delta(X) \delta(Y) $.
Using\,\eqref{eq-qplane}, we have
\begin{eqnarray*}
\delta(Y) \delta(X) 
& = & (X \otimes b + Y \otimes d) (X \otimes a + Y \otimes c) \\
& = & X^2 \otimes ba + YX \otimes da + XY \otimes bc + Y^2 \otimes dc .
\end{eqnarray*}
Similarly, 
\begin{eqnarray*}
\delta(X) \delta(Y) 
& = & (X \otimes a + Y \otimes c)  (X \otimes b + Y \otimes d) \\
& = & X^2 \otimes ab + YX \otimes cb + XY \otimes ad + Y^2 \otimes cd .
\end{eqnarray*}
Now using the defining relations of~$\SL_q(2)$ and the relation $YX = qXY$, we obtain
\begin{eqnarray*}
\delta(Y) \delta(X) - q \, \delta(X) \delta(Y)
& = & X^2 \otimes (ba-qab) + YX \otimes (da- qcb) \\ &&
+ XY \otimes (bc- qad) + Y^2 \otimes (dc-qcd) \\
& = & XY \otimes q (da-qcb + q^{-1}bc-ad) \\ 
& = & -XY \otimes q (ad - da- (q^{-1} - q)bc) = 0.
\end{eqnarray*}

The map $\delta$ being a morphism of algebras, it is enough to check its coassociativity and its counitarity
on the generators $X,Y$, which is easy to do.

(b) Let $\omega \in \CC_q[X,Y]$ be a coinvariant element, i.e. $\delta(\omega) = \omega \otimes 1$.
Recall the morphism of Hopf algebras $\pi: \SL_q(2) \to \CC_q[X,X^{-1},Y]$\index{$\CC_q[X,X^{-1},Y]$} 
of Lemma\,\ref{lem-noco-nococo}. The composed map 
\[
\delta' = (\id \otimes \, \pi) \circ \delta : \CC_q[X,Y] \to \CC_q[X,Y] \otimes \CC_q[X,X^{-1},Y]
\] 
turns the quantum plane~$\CC_q[X,Y]$ into a $\CC_q[X,X^{-1},Y]$-comodule algebra.
We have
$\delta'(\omega) = (\id \otimes \, \pi)(\omega \otimes 1) 
= \omega \otimes \pi(1) = \omega \otimes 1$.
Thus $\omega$ is coinvariant for the $\CC_q[X,X^{-1},Y]$-coaction. 
Now it follows from\,\eqref{eq-qplane} and the formula for~$\pi$ that
\begin{equation*}
\delta'(X) = X \otimes \pi(a) + Y \otimes \pi(c) = X \otimes X
\end{equation*}
and 
\begin{equation*}
\delta'(Y) = X \otimes \pi(b) + Y \otimes \pi(d) = X \otimes Y + Y \otimes X^{-1}.
\end{equation*}
Comparing with Formula\,\eqref{eq-copr-qLaurent} for the coproduct~$\Delta$ 
of the Hopf algebra $\CC_q[X,X^{-1},Y]$, 
we see that $\delta'$ is the restriction of~$\Delta$ to the subalgebra~$\CC_q[X,Y]$.
It follows from this remark and from Example\,\ref{exa-H=comodalg} that 
$\omega$ is a scalar multiple of the unit of $\CC_q[X,X^{-1},Y]$, which is also the unit of~$\CC_q[X,Y]$.
\qed
\end{proof}

\begin{exercise}
Let $q$ be a non-zero complex number.
For any integer $r >0$ define the \emph{$q$-integer}\index{$q$-integer}~$[r]$ by 
\begin{equation*}
[r] = 1 + q + \cdots + q^{r-1} = \frac{q^r - 1}{q-1}.
\end{equation*}
and the \emph{$q$-factorial}\index{$q$-factorial} $[r]!$ by
\begin{equation*}
[r]! = \prod_{k = 1}^r \, [k] = \frac{(q-1)(q^2-1) \cdots (q^r - 1)}{(q-1)^r}.
\end{equation*}
We agree that $[0]! = 1$. For $0 \leq r \leq n$ we define the \emph{$q$-binomial coefficient}\index{$q$-binomial coefficient}
\begin{equation*}
\left[
\begin{matrix}
n \\ r
\end{matrix}
\right]
= \frac{[n]!}{[r]! \, [n-r]!}.
\end{equation*}

(a) For $0 < r < n$ show the following $q$-analogue of the \emph{Pascal identity}\index{$q$-Pascal identity}
\begin{equation*}
\left[
\begin{matrix}
n \\ r
\end{matrix}
\right]
= \left[
\begin{matrix}
n-1 \\ r-1
\end{matrix}
\right]
+ q^r
\left[
\begin{matrix}
n-1 \\ r
\end{matrix}
\right] .
\end{equation*}

(b) Let $X,Y$ be variables subject to the relation $YX = qXY$. 
Prove the \emph{$q$-binomial formula}\index{$q$-binomial formula}
\begin{equation*}
(X+Y)^n = \sum_{r=0}^n \, \left[
\begin{matrix}
n \\ r
\end{matrix}
\right]
X^r Y^{n-r} .
\end{equation*}
\end{exercise}

\begin{exercise}\label{exo-qplane-coaction}
Recall the basis $\{X^i Y^j\}_{i,j \in \NN}$ of the quantum plane\index{quantum plane}~$\CC_q[X,Y] $.
Compute $\delta(X^i Y^j)$ for the coaction\,\eqref{eq-qplane}. 
\end{exercise}

\subsection{Quantum homogeneous spaces}\label{ssec-qhomogene}

Let $G$ be an algebraic group and $G'$ be an algebraic subgroup. To this data we associate the 
\emph{homogeneous space}~$G/G'$, 
whose elements are the left cosets~$gG'$ of~$G'$ in~$G$ with respect to $g\in G$;
in other words, two elements $g_1, g_2 \in G$ represent the same element of~$G/G'$
if and only if there exists $g' \in G'$ such that $g_2 = g_1 g'$.

To the inclusion $i : G' \hookrightarrow G$ corresponds
the morphism of Hopf algebras $\pi = i^*: \OO(G) \to \OO(G')$, which sends a function $u\in \OO(G)$
to its restriction to~$G'$. The map~$\pi$ is surjective. 
The composition 
\[
\delta = (\id \otimes \, \pi) \circ \Delta : \OO(G) \to \OO(G) \otimes \OO(G')
\]
turns $\OO(G)$ into an $\OO(G')$-comodule algebra. Let us consider the subalgebra
\[
\OO(G)^{\co-\OO(G')} \subset \OO(G)
\]
of coinvariant elements\index{$A^{\co-H}$}\index{coinvariant}. 

\begin{lemma}
An element $u\in \OO(G)$ belongs to the subalgebra $\OO(G)^{\co-\OO(G')}$ if and only if
$u(gg') = u(g)$ for all $g\in G$ and $g'\in G'$.
\end{lemma}

\begin{proof}
Identifying $\OO(G) \otimes \OO(G')$ with $\OO(G\times G')$ and using Formula\,\eqref{eq-defOG}
for the coproduct of~$\OO(G)$, we see that the above coaction~$\delta$ sends an element
$u\in \OO(G)$ to the function $\delta(u) \in \OO(G\times G')$ given by
\[
\delta(u)(g,g') = u(gg')
\]
for all $g\in G$ and $g'\in G'$. Such an element~$u$ is coinvariant if and only if $\delta(u) = u\otimes 1$,
which is equivalent to $\delta(u)(g,g') = u(g)1$ for all $g\in G$ and $g'\in G'$.
\qed
\end{proof}

It follows from the lemma and the above description of~$G/G'$
that the subalgebra~$\OO(G)^{\co-\OO(G')}$ of coinvariant elements
can be identified with the coordinate algebra\index{coordinate algebra}~$\OO(G/G')$ 
of the homogeneous space~$G/G'$.

The non-commutative analogue of a homogeneous space is the following.
Let $\pi : H \to \bar{H}$ be a surjective morphism of Hopf algebras.
The map 
\[
\delta = (\id \otimes \, \pi) \circ \Delta : H \to H \otimes \bar{H}
\]
turns $H$ into an $\bar{H}$-comodule algebra.
Let us consider the subalgebra~$H^{\co-\bar{H}}$ of coinvariant elements; by analogy with the previous classical case
we call~$H^{\co-\bar{H}}$ 
a \emph{quantum homogeneous space}.\index{quantum homogeneous space}

This general construction provides many examples of quantum homogeneous spaces;
see \cite{BM1,DGH,Ha1,HM,LPR,LS,Po,Sc1}.
We have already encountered such a situation with the surjective morphism of Hopf algebras
$\pi : \SL_q(2) \to \CC_q[X,X^{-1},Y]$ in Sect.\,\ref{ssec-noco-nococo},
where $\CC_q[X,X^{-1},Y]$ has been hinted at as a quantization of the coordinate algebra of 
the subgroup~$B$ of upper triangular matrices in~$SL_2(\CC)$.
It is well known that the homogeneous space $SL_2(\CC)/B$ is in bijection 
with the \emph{projective line}~$\CC\PP^1$.
Therefore the subalgebra $\SL_q(2)^{\co-\CC_q[X,X^{-1},Y]}$ can be seen 
as a quantization\index{quantization}\index{quantum projective line} of~$\CC\PP^1$.

\section{Hopf Galois extensions}\label{sec-HopfGalois}

It was noticed in the 1990's (see\,\cite{BM1, Du, Sc1}) 
that the right non-commutative version of a principal fiber bundle is the concept of a Hopf Galois extension, 
a notion which had been introduced in the 1960's by algebraists 
in order to extend the classical Galois theory of field extensions to a more general framework.

Let us now define Hopf Galois extensions. 
The use of the word ``Galois'' in this expression will be justified by Example\,\ref{exa-Galoisext} below.

\subsection{Definition and examples}\label{ssec-HopfGalois}

\begin{definition}\label{def-HopfGalois}
Let $H$ be a Hopf algebra\index{Hopf algebra} and $B$ an (associative unital) algebra.
An \emph{$H$-Galois extension}\index{Hopf Galois extension} of~$B$ is 
an $H$-comodule algebra\index{comodule algebra}\index{algebra!comodule}~$A$
with coaction\index{coaction}~$\delta: A \to A \otimes H$ such that
the following three conditions hold:
\begin{enumerate}
\item[(i)] 
$A$ contains $B$ as a subalgebra;

\item[(ii)] 
$B = A^{\co-H} = \{ a \in A \, |\, \delta(a) = a\otimes 1\}$;\index{$A^{\co-H}$}

\item[(iii)] 
the linear map
\begin{equation}\label{eq-beta}
\beta: A \otimes A \to A \otimes H \, ; \;\; a \otimes a' \mapsto (a\otimes 1) \, \delta(a') 
\end{equation}
induces a linear isomorphism $A \otimes _B A \overset{\cong}{\longrightarrow} A \otimes H$.
\end{enumerate}
\end{definition}

Let us comment on Condition\,(iii).
Firstly, the vector space $A \otimes _B A$ is by definition the quotient of~$A \otimes A$ 
by the subspace~$U$ spanned by all tensors of the form
\begin{equation*}
ab \otimes a' - a \otimes ba' . \qquad (a,a'\in A, \, b \in B)
\end{equation*}
Condition\,(iii) implies that the map~$\beta$ factors through the quotient space~$A \otimes _B A$.
Let us check this: it is enough to verify that $\beta$ vanishes on the generators of the subspace~$U$. Indeed,
\begin{eqnarray*}
\beta(ab \otimes a' - a \otimes ba')
& = & (ab \otimes 1) \,\delta(a') - (a \otimes 1) \, \delta(ba') \\
& = & (a \otimes 1)(b \otimes 1) \,\delta(a') - (a \otimes 1)\, \delta(b)\, \delta(a')
= 0
\end{eqnarray*}
in view of the fact that $b$ is coinvariant, hence satisfies $\delta(b) = b \otimes 1$.

The map~$\beta$ in Condition\,(iii) is the non-commutative analogue of the map
$\gamma: G \times P \to P \times P$ defined by\,\eqref{eq-gamma},
and the isomorphism $A \otimes _B A \overset{\cong}{\longrightarrow} A \otimes H$ is the non-commutative
analogue of the bijection $\gamma: G \times P \to P \times_X P$.
For this reason a Hopf Galois extension can be seen as 
a \emph{non-commutative principal fiber bundle}\index{principal fiber bundle!non-commutative}.

\begin{remark}
Let $A$ be an $H$-Galois extension of~$B$. 
Observe that, if $\dim\ A$ is finite, then so are $\dim\ A \otimes A$ and $\dim\ A \otimes_B A$.
In view of the isomorphism $A \otimes _B A \cong A \otimes H$, we deduce that the Hopf algebra~$H$
is finite-dimensional and that $\dim\ H \leq  \dim\ A$. 
If in addition $B = \CC$ is the ground field, then $A \otimes _B A = A \otimes A$ and $\dim\ H = \dim\ A$.
\end{remark}

\begin{remark}\label{rem-fflat}
Sometimes in the definition of an $H$-Galois extension $A$ of~$B$ one also requires 
$A$ to be \emph{faithfully flat} as a left $B$-module.
This means that taking the tensor product $\otimes_B M$ with a sequence 
of right $B$-modules produces an exact sequence if and only if the original sequence is exact. 
Finite-rank free or projective modules are examples of faithfully flat modules.
The Hopf Galois extensions we will consider in Sect.\,\ref{sec-versal} satisfy this extra condition.
\end{remark}

According to\,\cite[Sect.\,7]{CS}, Definition\,\ref{def-HopfGalois} was introduced to give a generalization of 
Galois theory to arbitrary commutative rings, 
the finite group of automorphisms in the classical theory being replaced by a Hopf algebra.

Let us now present the prototypical example of a Hopf Galois extension, which justifies the terminology used.

\begin{example}\label{exa-Galoisext}
If $K$ is a finite \emph{Galois extension}\index{Galois extension} of a number field~$k$
with Galois group~$G$, then by Proposition\,\ref{prop-G-alg}\,(a) the field~$K$ is an $\OO(G)$-comodule $k$-algebra
with coaction~$\delta$ given for all $a\in K$ by
\begin{equation*}
\delta(a) = \sum_{g\in G} \, ga \otimes \delta_g.
\end{equation*}
We know that the subalgebra of coinvariant elements of~$K$ is the subalgebra of $G$-invariant elements,
therefore coinciding with the field~$k$.
The map 
\begin{equation*}
\beta: K \otimes_k K \to K \otimes_k \OO(G)
\end{equation*}
defined by\,\eqref{eq-beta} is an isomorphism (see e.g. \cite[Sect.\,8.1.2]{Mo}). 
Therefore, $K$ is an $\OO(G)$-Galois extension of~$k$. 
\end{example}

Here are more examples of Hopf Galois extensions.

\begin{example}
If $P \to X$ is a \emph{principal $G$-bundle}\index{principal fiber bundle}, then
$\OO(P)$ is an $\OO(G)$-Galois extension of~$\OO(X)$\index{coordinate algebra}.
\end{example}

\begin{example}
Let $A = \CC[x,x^{-1}]$ be the algebra of Laurent polynomials in one variable and let $n \geq 1$ be an integer.
We can give~$A$ a $\ZZZ/n$-grading by setting $\deg (x^i) \equiv i \pmod{n}$. This is a strong grading 
in the sense defined above. 
The algebra~$A$ becomes a $\CC[\ZZZ/n]$-Galois extension of the subalgebra~$B= \CC[x^n,x^{-n}]$.
This is the algebraic version of the principal $\ZZZ/n$-bundle $\pi_n: S^1 \to S^1$
of Example\,\ref{exa-ppal}.
\end{example}

\begin{example}\label{exa-Galois}
\emph{(Strongly graded algebras)}
Let $G$ be a group. We know (see Proposition\,\ref{prop-graded}) that any $G$-graded algebra~$A$ is
a $\CC[G]$-comodule algebra. Recall that the subalgebra of coinvariants is the $e$-compotent~$A_e$.
Such a comodule algebra is a $\CC[G]$-Galois extension of~$A_e$ if and only if $A$ is a
\emph{strongly $G$-graded algebra}\index{strongly graded algebra}\index{algebra!strongly graded}
\index{graded algebra}\index{algebra!graded}, 
i.e. a $G$-graded algebra such that $A_g A_h = A_{gh}$
for all $g,h \in G$ (see\,\cite[Th.\,8.1.7]{Mo}).

The matrix algebra $M_N(\CC)$ with the $\ZZZ/N$-grading given in Example\,\ref{exa-graded-c}\,(b)
and the algebra of quaternions\index{quaternion algebra} 
with the $(\ZZZ/2)^2$-grading of Example\,\ref{exa-graded-d} are strongly graded algebras.
\end{example}

\begin{remark}
In classical differential geometry once one has a principal $G$-bundle, 
one can construct a vector bundle associated with it and with an additional representation of~$G$, 
equip this vector bundle with a connection, and derive various characteristic classes. 
Nowadays these classical constructions have non-commutative counterparts;
for details, see\,\cite{BM1, DGH, Ha1, HM, Pf, Wo2}.
\end{remark}

\subsection{The classification problem}\label{ssec-class}\index{classification}

We say that two $H$-Galois extensions $A,A'$ of~$B$ are \emph{isomorphic}\index{Hopf Galois extension!isomorphism of}
if there is an isomorphism of $H$-comodule algebras $A \to A'$.

In Sect.\,\ref{ssec-functor} (see Corollary\,\ref{coro-ppal})
we showed how to classify principal $G$-bundles:
there exists a bijection
\[
[X,BG] \overset{\cong}{\longrightarrow} \Iso_G(X)
\]
which is functorial in~$X$. 
Recall that $\Iso_G(X)$ is the set of homeomorphism classes of principal $G$-bundles with base space~$X$
and $[X,BG]$ is the set of homotopy classes of continuous maps from~$X$ to~$BG$.

We wish likewise to classify all $H$-Galois extensions of~$B$ up to isomorphism for a
given Hopf algebra~$H$ and a given algebra~$B$. In other words,
we would like to compute the set~$\Gal_H(B)$\index{$\Gal_H(B)$} 
of isomorphism classes of $H$-Galois extensions of~$B$. 

So far not many general results on~$\Gal_H(B)$ are available. Here is one.

\begin{theorem}
The set~$\Gal_H(B)$ is non-empty.
\end{theorem}

This is a consequence of the following result.

\begin{proposition}
The tensor product algebra $A = B \otimes H$ is an $H$-Galois extension of~$B = B \otimes 1$ with coaction 
$\delta = \id_B \otimes \, \Delta: A = B \otimes H \to A \otimes H = B \otimes H \otimes H$,
where $\Delta$ is the coproduct of~$H$.
\end{proposition}

This Hopf Galois extension is called the \emph{trivial Hopf Galois extension}\index{Hopf Galois extension!trivial}.
Its isomorphism class is thus a special point of~$\Gal_H(B)$, 
just as the trivial principal $G$-bundle\index{principal fiber bundle!trivial}
is a special element of the set~$\Iso_G(X)$ of homeomorphism classes of principal $G$-bundles with given base space~$X$.

\begin{proof}
The map $\delta$ turns $A$ into an $H$-comodule algebra. Proceeding as in Example\,\ref{exa-H=comodalg},
we prove that the subalgebra of coinvariant elements coincides with~$B \otimes 1 = B$.

Finally we have to establish that the map $\beta: A \otimes_B A \to A \otimes H$ of~\eqref{eq-beta} is
an isomorphism. Now 
\[
A \otimes_B A = (B \otimes H) \otimes_B (B \otimes H) = B \otimes H \otimes H
\]
and $A \otimes H = B \otimes H \otimes H$.
It suffices to check that the map $\beta_1: H \otimes H \to H \otimes H$ defined for all $x,y \in H$ by
\begin{equation*}
\beta_1(x\otimes y) =  (x\otimes 1) \, \Delta(y) = \sum_{(y)} \, x y_{(1)} \otimes y_{(2)}
\end{equation*}
is a linear isomorphism
(here again we use the Heyneman--Sweedler sigma notation\index{Heyneman--Sweedler notation} 
of Sect.\,\ref{ssec-Sweedler-notation}). 
Define a map $\beta_2$ in the other direction by
\begin{equation*}
\beta_2(x\otimes y) =  (x\otimes 1) (S \otimes \id)(\Delta(y)) = \sum_{(y)} \, xS(y_{(1)}) \otimes y_{(2)} .
\end{equation*}
On one hand, by \eqref{HS-anti} and\,\eqref{HS-coun} we have
\begin{eqnarray*}
(\beta_1 \circ \beta_2)(x\otimes y)
& = &  \sum_{(y)} \, x S(y_{(1)}) y_{(2)} \otimes y_{(3)} 
=  \sum_{(y)} \, x \eps(y_{(1)}) \otimes y_{(2)} \\
& = & x \otimes  \sum_{(y)} \, \eps(y_{(1)}) y_{(2)} 
= x \otimes y,
\end{eqnarray*}
which proves $\beta_1 \circ \beta_2 = \id_{H \otimes H}$.
On the other,
\begin{eqnarray*}
(\beta_2 \circ \beta_1)(x\otimes y) 
& = &  \sum_{(y)} \, x y_{(1)} S(y_{(2)}) \otimes y_{(3)} 
=  \sum_{(y)} \, x \eps(y_{(1)}) \otimes y_{(2)} \\
& = & x \otimes  \sum_{(y)} \, \eps(y_{(1)}) y_{(2)} 
= x \otimes y.
\end{eqnarray*}
This completes the proof of the bijectivity of~$\beta_1$, hence of~$\beta$.
\qed
\end{proof}

\subsection{The set $\Gal_H(\CC)$ may be non-trivial}\label{ssec-Gal-obj}

We observed in Sect.\,\ref{ssec-fiber} that any fiber bundle over a point is trivial. 
The corresponding result for $H$-Galois extensions of the ground field~$\CC$ may not hold.
To show this let us present examples of Hopf algebras~$H$ for which $\card \Gal_H(\CC) >1$\index{$\Gal_H(\CC)$}.

It is convenient to introduce the following definition.

\begin{definition}\label{def-HGal}
Let $H$ be a Hopf algebra. An \emph{$H$-Galois object}\index{Galois object} 
is an $H$-Galois extension of~$\CC$.
\end{definition}

\subsubsection{The case of a group algebra}\label{sssec-Gal-group}\index{group algebra}\index{algebra!group}

Let us consider $H = \CC[G]$ for some group~$G$. 
We now describe $\Gal_H(\CC)$ for this Hopf algebra.

By Example\,\ref{exa-Galois} we know that any $\CC[G]$-Galois extension~$A$ of~$\CC$ is a strongly $G$-graded
algebra $A = \bigoplus_{g \in G}\, A_g$ such that $A_e = \CC$. Since it is strongly graded, it follows that each
component~$A_g$ is one-dimensional. 
Let us pick a non-zero element $u_g$ in each~$A_g$. Then
the product structure of the algebra~$A$ is determined by the products $u_g u_h$ for each pair $(g,h)$ of elements  
of~$G$. We have
\begin{equation}\label{eq-cochain}
u_g u_h  = \lambda(g,h) \, u_{gh} \in A_{gh}
\end{equation}
for some scalar $\lambda(g,h)$ depending on~$g$ and~$h$. 
Such a scalar is non-zero since by definition the multiplication map
$A_g \times A_h \to A_{gh}$ is surjective. Thus, the family of scalars~$\lambda(g,h)$
defines a map $\lambda: G \times G \to \CC^{\times}$, where $\CC^{\times} = \CC \setminus\{ 0\}$.

The map $\lambda$ satisfies an additional relation called \emph{cocyclicity}\index{cocycle}, originating from the fact that
the product\index{product} of~$A$ is associative.\index{algebra!associative}
Indeed, we have $(u_g u_h) u_k = u_g (u_h u_k)$ for all $g,h,k \in G$. Using\,\eqref{eq-cochain},
we obtain the following equality
\begin{equation}\label{eq-cocycle}
\lambda(g,h) \, \lambda(gh,k) = \lambda(h,k) \, \lambda(g,hk)
\end{equation}
for all $g,h,k \in G$. A map $\lambda: G \times G \to \CC^{\times}$ satisfying the identity\,\eqref{eq-cocycle}
is called a \emph{$2$-cocycle}\index{cocycle} for the group~$G$. 

It can be checked (see any textbook on group cohomology, for instance\,\cite{Br})
that the pointwise multiplication of maps from $G \times G$ to~$\CC^{\times}$ induce an abelian group structure
on the set~$Z^2(G, \CC^{\times})$ of $2$-cocycles for~$G$.

Let us choose another non-zero element $v_g$ in each~$A_g$. Then we have $v_g = \mu(g) \, u_g$ for some non-zero 
scalar~$\mu(g)$. Combining this with\,\eqref{eq-cochain}, we obtain
$v_g v_h  = \lambda'(g,h) \, v_{gh}$, where 
\begin{equation}\label{eq-cobound}
\lambda'(g,h)  = \frac{\mu(g)\mu(h)}{\mu(gh)} \, \lambda(g,h)
\end{equation}
for all $g,h \in G$. We say that two $2$-cocycles $\lambda, \lambda'$ are \emph{cohomologous} if they are related
by an equation of the form\,\eqref{eq-cobound}.
It is easy to check that for any map $\mu: G \to \CC^{\times}$ the assignment $(g,h) \mapsto \mu(g)\mu(h)/\mu(gh)$
is a $2$-cocycle, which we call a \emph{coboundary}\index{coboundary}. 
Moreover, the set~$B^2(G, \CC^{\times})$ of coboundaries is a subgroup of~$Z^2(G, \CC^{\times})$.

We define the \emph{second cohomology group}\index{cohomology group}\index{group cohomology} of~$G$ as the quotient
\begin{equation*}
H^2(G,\CC^{\times}) = Z^2(G, \CC^{\times}) / B^2(G, \CC^{\times}).
\end{equation*}
It follows from the previous arguments that we have a bijection
\begin{equation}\label{eq-H2}
\Gal_{\CC[G]}(\CC) \cong H^2(G,\CC^{\times}) .
\end{equation}

\begin{example}\label{exa-cyclic}
It is well known (see\,\cite[V.6]{Br}) that for a cyclic group~$G$ (infinite or not) we have 
$H^2(G,\CC^{\times}) = 0$;
for such a group $\Gal_{\CC[G]}(\CC)$ is then trivial by\,\eqref{eq-H2}, i.e. any $\CC[G]$-Galois object is trivial.
\end{example}

\begin{example}
Let $G = (\ZZZ/N)^r$ for some integer $r\geq 2$. Then 
\[
H^2(G,\CC^{\times}) \cong (\ZZZ/N)^{r(r-1)/2},
\]
which implies that $\Gal_{\CC[G]}(\CC) > 1$ for such a group. 
This is of course a rather surprising result, which again shows that non-commutative geometry has features
which classical geometry does not have.
\end{example}

\begin{example}
Even more surprising, 
if $G = \ZZZ^r$ is the free abelian group of rank~$r \geq 2$, then 
\[
H^2(G,\CC^{\times}) \cong (\CC^{\times})^{r(r-1)/2}.
\]
Hence, for $r\geq 2$ there are \emph{infinitely many} isomorphism classes of $\CC[\ZZZ^r]$-Galois objects\index{Galois object}.
\end{example}

\begin{remark}
In contrast with Example\,\ref{exa-cyclic}, 
the cohomology group $H^2(\ZZZ/2, \RR^{\times})$ of the cyclic group of order~$2$, now with coefficients 
in $\RR^{\times} = \RR \setminus \{0\}$, is not trivial: 
\[
H^2(G,\RR^{\times}) = \RR^{\times}/(\RR^{\times})^2 \cong \ZZZ/2.
\]
Proceeding as above, we deduce that, up to isomorphism, there are two real $\ZZZ/2$-Galois extensions of~$\RR$.
The trivial one is $\RR[\ZZZ/2] = \RR[x]/(x^2 -1) \cong \RR \times \RR$,
which has zero divisors.
The second one is the field~$\CC = \RR[x]/(x^2 +1)$ of complex numbers. 
Both are two-dimensional superalgebras\index{superalgebra}, 
with the even part spanned by the unit~$1$ and the odd part by the image of~$x$.
\end{remark}

\begin{remark}
Group algebras are cocommutative Hopf algebras and by\,\eqref{eq-H2} 
the group $\Gal_H(\CC)$ is abelian in this case. 
More generally, for any \emph{cocommutative}\index{Hopf algebra!cocommutative} Hopf algebra~$H$, 
the set~$\Gal_H(\CC)$ has the structure of an abelian group;
its product is induced by the cotensor product\footnote{The concept of the cotensor product of comodules was first
introduced in\,\cite{EM}. See also \cite{Mo, Sw}.} of comodule algebras (see for example \cite[10.5.3]{Ca}).
\end{remark}

\subsubsection{Taft algebras}\label{sssec-Taft}

Let $N$ be an integer $\geq 2$ and $q$ a root of unity\index{root of unity} of order~$N$.
The \emph{Taft algebra}\index{Taft algebra}\index{algebra!Taft} of dimension~$N^2$ is
the algebra $H_{N^2}$ generated by two generators $g$, $x$ subject to the relations
\begin{equation*}
g^N = 1, \quad x^N = 0 , \quad xg = q \, gx.
\end{equation*}
It is a Hopf algebra with 
\begin{equation*}
\Delta(g) = g \otimes g, \quad 
\Delta(x) = 1 \otimes x + x \otimes g, \quad 
\eps(g) = 1, \quad 
\eps(x) = 0.
\end{equation*}
This Hopf algebra is neither commutative, nor cocommutative\index{Hopf algebra!non-commutative non-cocommutative}.
When $N=2$, the four-dimensional Hopf algebra~$H_4$ is known under the name 
of \emph{Sweedler algebra}\index{Sweedler algebra}\index{algebra!Sweedler}.

For any $s \in \CC$ consider the algebra
\begin{equation*}
A_s = \CC \, \langle \, G,X\, \rangle /
\left(G^N - 1, \, X^N - s , \, XG - q \, GX \right) .
\end{equation*}
It is a right $H_{N^2}$-Galois object\index{Galois object} with coaction given by
\begin{equation*}
\Delta(G) = G \otimes g, \quad 
\Delta(X) = 1 \otimes x + X \otimes g.
\end{equation*}
By\,\cite[Prop.\,2.17 and Prop.\,2.22]{Ma1} (see also\,\cite{DT2})
any $H_{N^2}$-Galois object is isomorphic to~$A_s$ for some scalar~$s$,
and any two such Galois objects $A_s$ and~$A_t$ are isomorphic if and only if $s=t$.
Therefore,\index{$\Gal_H(\CC)$}
\begin{equation*}
\Gal_{H_{N^2}}(\CC) \cong \CC ,
\end{equation*}
which is an abelian group although the Hopf algebra~$H_{N^2}$ is not cocommutative.

See also\,\cite{Bi1, Bi2, Ne, PvO} for the determination of $\Gal_H(\CC)$ for other finite-dimen\-sion\-al
Hopf algebras~$H$ generalizing the Sweedler algebra.

\subsubsection{The quantum enveloping algebra $U_q\, \gog$}\index{quantum enveloping algebra}

Masuoka\,\cite{Ma3} determined $\Gal_H(\CC)$ when $H = U_q\, \gog$ is Drinfeld--Jimbo's 
quantum enveloping algebra mentioned in Sect.\,\ref{ssec-qea}, Remark\,\ref{rem-Uqg}.
A partial result had been given in\,\cite[Th.\,4.5]{KS} under the form of a surjection
\begin{equation*}
\Gal_H(\CC) \twoheadrightarrow H^2(\ZZZ^r, \CC^{\times}) \cong (\CC^{\times})^{r(r-1)/2},
\end{equation*}
where $r$ is the size of the corresponding Cartan matrix (see also\,\cite{Au1}).

\subsection{Push-forward of central Hopf Galois extensions}\label{ssec-alg-functor}

In Sect.\,\ref{ssec-functor} we saw that, given a continuous map $\varphi: X' \to X$, there is a functorial
map
\begin{equation*}
\varphi^* : \Iso_G(X) \to \Iso_G(X')
\end{equation*}
induced by $P \mapsto \varphi^*(P)$. 

In our algebraic setting we may wonder whether, given a Hopf algebra~$H$
and a morphism of algebras $f: B \to B'$, there exists a functorial map
\begin{equation*}
f_* : \Gal_H(B) \to \Gal_H(B')
\end{equation*}
which would be the algebraic analogue 
of the pull-back of bundles\index{principal fiber bundle!pull-back of}\index{pull-back}.
The most natural way to construct such 
a \emph{push-forward}\index{push-forward}\index{Hopf Galois extension!push-forward of} map~$f_*$ is the following.
Let $A$ be an $H$-Galois extension of~$B$. 
Since $B$ is a subalgebra of~$A$, we can consider $A$ as a left $B$-module.
Given a morphism of algebras $f: B \to B'$, we can then define the left $B'$-module~$f_*(A)$ as
\begin{equation*}
f_*(A) = B' \otimes_B A.
\end{equation*}
Here we have used the fact that $B'$ is a right $B$-module via the morphism of algebras~$f$.
It is clear that if $g: B' \to B''$ is another morphism of algebras, then 
we have a natural isomorphism $(g\circ f)_*(A) \cong g_*(f_*(A))$ of $B''$-modules.

There is however a serious problem with this construction: in general $f_*(A) = B' \otimes_B A$ is not an algebra!
To circumvent this difficulty, we will restrict to
\emph{central $H$-Galois extensions}\index{Hopf Galois extension!central}, 
namely to those for which $B$ is contained in the center of~$A$; this implies of course that $B$ is a commutative algebra
(central Hopf Galois extensions were first discussed in\,\cite{Ru}).
The algebra~$\AA_H$ defined in Sect.\,\ref{sssec-total-alg} below is an (important) example of a central $H$-Galois extension.

We denote by $\ZGal_H(B)$\index{$\ZGal_H(B)$} the set of isomorphism classes of central $H$-Galois extensions of~$B$.
Then a morphism of \emph{commutative}\index{algebra!commutative} algebras $f: B \to B'$
induces a push-forward map $f_* : \ZGal_H(B) \to \ZGal_H(B')$ given by $A \mapsto f_*(A)$
and satisfying the desired functorial properties\footnote{For this to hold we need the extra faithful flatness condition
mentioned in Sect.\,\ref{ssec-HopfGalois}, Remark\,\ref{rem-fflat}.} (see\,\cite{Ka0, KS}).

In particular, let $\chi: B \to \CC$ be a character of~$B$. Then $A \mapsto \chi_*(A)$ induces a map
$\chi_* : \ZGal_H(B) \to \ZGal_H(\CC)$.
Observe that $\ZGal_H(\CC) = \Gal_H(\CC)$ when $B = \CC$ is the ground field, as the latter is always central.
In analogy with the case of a fiber bundle (see Exercise\,\ref{exo-bundle}\,(a)),
we call $\chi_*(A) = \CC \otimes_B A$ the \emph{fiber}\index{fiber bundle!fiber of} of the $H$-Galois extension~$A$ at~$\chi$.
Note that $\chi_*(A) = A / \mm A$, where $\mm$ is the kernel of~$\chi$.

\subsection{Universal central Hopf Galois extensions}\label{ssec-universal}

A non-commutative analogue of the classifying space\index{classifying space}~$BG$ mentioned in Sect.\,\ref{ssec-functor}
would be a central $H$-Galois extension~$\AA_H$\index{$\AA_H$} of some commutative algebra~$\BB_H$\index{$\BB_H$}
such that for any commutative algebra~$B$ and any central $H$-Galois extension~$A$ of~$B$
there exists a morphism of algebras $f: \BB_H \to B$ such that $f_*(\AA_H) \cong A$.
In other words, we would have a functorial surjection
\begin{equation*}
\Alg(\BB_H,B) \twoheadrightarrow \ZGal_H(B)
\end{equation*}
induced by $f \mapsto f_*(\AA_H)$.
Here $\Alg(\BB_H,B)$ is the set of morphisms of algebras from~$\BB_H$ to~$B$.

Does such a central $H$-Galois extension~$\AA_H$ exist for an arbitrary Hopf algebra~$H$?
It is an open question.
We do not even know whether in general there exists a central $H$-Galois extension~$\BB_H \subset \AA_H$
with a natural surjection\index{character of an algebra} 
\begin{equation*}
\Alg(\BB_H,\CC) \twoheadrightarrow \ZGal_H(\CC) = \Gal_H(\CC)
\end{equation*}
from the set of characters\index{character of an algebra} of~$\BB_H$ 
to the set of isomorphism classes of $H$-Galois objects.
If such a surjection existed and was even bijective, then the $H$-Galois objects\index{Galois object} 
would be classified up to isomorphism by the characters of~$\BB_H$.

\begin{example}
Let us give an example for which $H$-Galois objects can be classified by the characters of a commutative algebra~$\BB$.
Take the Taft algebra~$H_{N^2}$ introduced in~Sect.\,\ref{sssec-Taft}.
Let $\BB$ be the polynomial algebra~$\CC[s]$ and $\AA = A_s$ considered as a $\CC[s]$-module,
where $A_s$ is the Galois object defined in \emph{loc.\ cit.} Each complex number~$s$ gives rise to a unique
character $\chi_s$ of~$\CC[s]$; it is tautologically defined by $\chi(s) = s$.
The map $s \mapsto \chi_s$ induces a bijection $\CC \to \Alg(\CC[s], \CC) = \Alg(\BB, \CC)$.
Now the assignment $\chi_s \mapsto (\chi_s)_*(\AA)$ induces a bijection
\[
\Alg(\BB, \CC) \overset{\cong}{\longrightarrow} \Gal_{H_{N^2}}(\CC) .
\]
\end{example}

When in 2005 I lectured on Hopf Galois extensions at the \emph{XVIo Coloquio Latino\-americano de \'Algebra}
in Colonia del Sacramento, Uruguay, I raised the question of the existence of a universal central Hopf Galois extension.
Eli Aljadeff immediately suggested the use of an appropriate theory of 
polynomial identities, based on his joint work\,\cite{AHN} with Haile and Natapov on group-graded algebras. 
In\,\cite{AK} we implemented Aljadeff's idea, using a theory of polynomial identities
for comodule algebras. Given a Hopf algebra~$H$ and an $H$-comodule algebra~$A$, we constructed
a ``universal $H$-comodule algebra''~$\UU_H(A)$ out of these identities. Localizing~$\UU_H(A)$,
we obtained a central $H$-Galois extension~$\AA_H$ of some commutative algebra~$\BB_H$,
the latter being a nice domain. 
The Hopf Galois extension $\BB_H \subset \AA_H$ comes with a map of the form
\[
\Alg(\BB_H, \CC) \to \Gal_H(\CC)\, ; \;\; \chi \mapsto \chi_*(\AA_H).
\]

In the next section we will construct this central $H$-Galois extension directly,
without passing through polynomial identities. Nevertheless the reader interested in polynomial identities,
the universal $H$-comodule algebra~$\UU_H(A)$ and the precise connection with
the central $H$-Galois extension constructed in Sect.\,\ref{ssec-generic}, may learn the details from~\cite{AK,Ka15}.

\section{Flat deformations of Hopf algebras}\label{sec-versal}

\medskip
\begin{flushright}
\emph{De pronto me sent{\'\i} pose{\'\i}do
por un aura\\ de inspiraci\'on
que me permiti\'o improvisar\\
respuestas cre{\'\i}bles y chiripas milagrosas.\\ 
Salvo en las matem\'aticas,
que no se me\\ rindieron ni en lo que Dios quiso.} \cite{GM}\\
\end{flushright}
\medskip

Let $H$ be a Hopf algebra.
The aim of this final section is to construct the commutative algebra~$\BB_H$
and the central $H$-Galois extension~$\AA_H$ of~$\BB_H$ we have just mentioned. 
When $H$ is finite-dimensional, the algebra~$\BB_H$ is the coordinate algebra of a smooth algebraic variety whose
dimension is equal to~$\dim\ H$. 
The algebra~$\AA_H$ is a deformation\index{deformation} of~$H$ as an $H$-comodule algebra;
this deformation is parametrized by the characters of~$\BB_H$.

We conclude these notes by showing how to apply these constructions to the quantum enveloping algebra~$U_q\,\gs\gl(2)$ 
and to its finite-dimensional quotients~$\gu_d$.

\subsection{A universal construction by Takeuchi}\label{sec-Takeuchi}

Let $C$ be a \emph{coalgebra}\index{coalgebra}, 
that is a vector space equipped with two linear maps
$\Delta : C \to C \otimes C$ (called the \emph{coproduct}\index{coproduct}) and 
$\eps: C \to \CC$ (called the \emph{counit}\index{counit})
satisfying the coassociativity\index{coassociativity} identity\,\eqref{eq-coassoc}
and the counitality\index{counitality} identity\,\eqref{eq-counit}.
There is a coalgebra underlying any bialgebra or any Hopf algebra.

Takeuchi\,\cite[Chap.\,IV]{Ta} proved the following result.

\begin{theorem}\label{th-Tak}
Given a coalgebra~$C$, 
there exist a commutative Hopf algebra~$\SS_C$\index{$\SS_H$} and a morphism of coalgebras $t: C \to \SS_C$
such that for any morphism of coalgebras $f : C \to H'$ to a commutative Hopf algebra~$H'$
there is a unique morphism of Hopf algebras 
\[
\widetilde{f} : \SS_C \to H'
\]
satisfying $f = \widetilde{f} \circ t$. The Hopf algebra~$\SS_C$ is unique up to unique isomorphism.
\end{theorem}

We say that $\SS_C$ is the \emph{free commutative Hopf algebra}\index{Hopf algebra!free commutative} over the coalgebra~$C$.
It can be constructed as follows.

\subsubsection{Construction of~$\SS_C$}\label{sec-construction}

Pick a copy $t_C$ of the underlying vector space of~$C$, that is to say we assign a symbol~$t_x$ 
to each element $x\in C$ so that the map $x \mapsto t_x$ is linear and defines a linear isomorphism $t: C \to t_C$.
Let $\Sym(t_C)$ be the \emph{symmetric algebra}\index{symmetric algebra}\index{algebra!symmetric}
over the vector space~$t_C$. It means concretely the following:
if $\{x_i\}_{iÊ\in I}$ is a basis of~$C$, then $\Sym(t_C)$ is the algebra $\CC[t_{x_i}]_{ i\in I}$ 
of polynomials in the variables~$t_{x_i}$.

The commutative algebra~$\Sym(t_C)$ is a bialgebra\index{bialgebra} with coproduct and counit given 
on the generators~$t_x$ (in terms of the Heyneman--Sweedler notation)\index{Heyneman--Sweedler notation} by
\begin{equation}\label{tx-coprod}
\Delta(t_x) = \sum_{(x)} \, t_{x_{(1)}} \otimes t_{x_{(2)}}
\quad\text{and}\quad
\eps(t_x) = \eps(x). \qquad (x\in C)
\end{equation}

In general, the bialgebra $\Sym(t_C)$ does not have an antipode: 
indeed, if $x \in C$ is a group-like element\index{group-like element}, then by\,\eqref{def-grouplike}
we have $\Delta(t_x) = t_x \otimes t_x$ and $\eps(t_x) = 1$. If there existed an antipode~$S$, then it would follow
from the previous equalities and from\,\eqref{eq-anti} that $S(t_x) t_x = 1$, hence
$S(t_x)  = 1/t_x$, which is not a polynomial. 
But this computation gives us hope that we may turn the bialgebra~$\Sym(t_C)$ into a Hopf algebra 
by using rational algebraic fractions instead of mere polynomials. This can indeed be done thanks to the following fact.

Let us denote by $\Frac\Sym(t_C)$ the field of fractions of~$\Sym(t_C)$: if $\{x_i\}_{iÊ\in I}$ is a basis of~$C$, 
then $\Frac\Sym(t_C)$ is the algebra of rational algebraic fractions in the variables~$t_{x_i}$ ($i\in I$).
There exists a unique linear map $t^{-1}: C \to \Frac\Sym(t_C)$ such that 
\begin{equation*}
\sum_{(x)} \, t^{-1}_{x_{(1)}} t_{x_{(2)}} = \eps(x) 1
= \sum_{(x)} \, t_{x_{(1)}}  t^{-1}_{x_{(2)}}
\end{equation*}
for all $x\in C$ (for a proof, see\,\cite[Lemma\,A.1]{AK}).
Then the subalgebra of $\Frac\Sym(t_C)$ generated by all elements~$t_x$ and~$t^{-1}_x$ ($x \in C$)
satisfies the requirements of Theorem\,\ref{th-Tak} to be the free commutative Hopf algebra~$\SS_C$. 
This subalgebra is a Hopf algebra with coproduct and counit given by\,\eqref{tx-coprod}
and the additional formulas
\begin{equation*}
\Delta(t^{-1}_x) = \sum_{(x)} \, t^{-1}_{x_{(2)}} \otimes t^{-1}_{x_{(1)}}
\quad\text{and}\quad
\eps(t^{-1}_x) = \eps(x). \qquad (x\in C)
\end{equation*}
The antipode is given on the generators~$t_x$ and~$t_x^{-1}$ by 
\begin{equation*}
S(t_x) = t^{-1}_x
\quad\text{and}\quad
S(t^{-1}_x) = t_x .
\end{equation*}
To check the universal property in Theorem\,\ref{th-Tak}, define the morphism $\widetilde{f}: \SS_C \to H'$ 
by $\widetilde{f}(t_x) = f(x)$ and $\widetilde{f}(t^{-1}_x) = S'(f(x))$, 
where $S'$ is the antipode of~$H'$.

It follows by construction that $\SS_C$, being a subalgebra of some field of rational functions, is
a \emph{domain}, i.e. an algebra without zero divisors.

In the sequel we will apply Takeuchi's construction to the underlying coalgebra of an arbitrary Hopf algebra~$H$, 
thus leading to the commutative algebra~$\SS_H$.

\subsubsection{Pointed Hopf algebras}\label{ssec-pointed}

A Hopf algebra is \emph{pointed}\index{Hopf algebra!pointed} if any simple subcoalgebra
is one-dimen\-sional. Group algebras, Taft algebras, enveloping algebras of Lie algebras, 
Drinfeld-Jimbo quantum enveloping algebras~$U_q\, \gog$
and their quotients are examples of pointed Hopf algebras.

When $H$ is a pointed Hopf algebra, then the free commutative Hopf algebra~$\SS_H$ 
over the coalgebra underlying~$H$
has a simple description in terms of the group $\Gr(H)$\index{$\Gr(H)$}
of group-like elements\index{group-like element} introduced in Sect.\,\ref{ssec-Hopfalg}, namely
\begin{equation}\label{SH-pointed}
\SS_H = \Sym(t_H)\left[ \frac{1}{t_g} \right]_{g \in \Gr(H)}.
\end{equation}

\begin{example}
If $H = \CC[G]$ is a group algebra\index{group algebra}\index{algebra!group}, 
then $\Sym(t_H)$ is the polynomial algebra
\begin{equation*}
\Sym(t_H) = \CC[t_g]_{g\in G}.
\end{equation*}
Since $H$ is pointed and $\Gr(H) = G \subset \CC[G]$, then by\,\eqref{SH-pointed} 
the free commutative Hopf algebra~$\SS_H$ is the algebra of \emph{Laurent polynomials} on the symbols~$t_g$ ($g\in G$), 
or equivalently the algebra of the free abelian group~$\ZZZ^{(G)}$ generated by the symbols~$t_g$:
\begin{equation*}
\SS_H = \CC[t_g, t_g^{-1}]_{g\in G}  = \CC[\ZZZ^{(G)}].
\end{equation*}
\end{example}

\begin{example}
Let $G$ be a finite group and $H$ be the function algebra\index{function algebra}~$\OO(G)$ 
(this Hopf algebra is not pointed when $G$ is not abelian).
Then $\Sym(t_H) = \CC[t_g \, |\, g\in G]$ and
\begin{equation*}\label{eq-Theta}
\SS_H = \CC[t_g]_{g\in G} \left[ \frac{1}{\Theta_G} \right],
\end{equation*}
where $\Theta_G = \det (t_{gh^{-1}})_{g,h\in G}$ is \emph{Dedekind's group determinant}\index{group determinant}
(see\,\cite[App.\,B]{AK}).
\end{example}

\subsection{The generic Hopf Galois extension associated with a Hopf algebra}\label{ssec-generic}
\index{generic Hopf Galois extension}\index{Hopf Galois extension!generic}

In this section we associate with any Hopf algebra~$H$ a central $H$-Galois extension $\BB_H \subset \AA_H$,
where the ``base space''~$\BB_H$ is a nice commutative algebra whose size is related to the dimension of~$H$.
We can see~$\AA_H$ as a deformation\index{deformation} of~$H$ over the parameter space~$\BB_H$.

\subsubsection{The algebra~$\BB_H$}\label{sssec-base-alg}

Let $H$ be a Hopf algebra. In order to construct the ``base space''~$\BB_H$\index{$\BB_H$}
we apply Takeu\-chi's theorem to the situation where $C$ is the coalgebra underlying~$H$ and 
$H' = H_{\ab}$\index{$H_{\ab}$} is the largest commutative Hopf algebra quotient of~$H$: 
it is the quotient of~$H$ by the ideal generated by all commutators $xy - yx$ ($x,y \in H$). 

Let $\pi : H  \to H_{\ab}$ be the canonical Hopf algebra surjection. Then by Theorem\,\ref{th-Tak}, 
for the free commutative Hopf algebra\index{Hopf algebra!free commutative}~$\SS_H$\index{$\SS_H$}
there exists a unique morphism of Hopf algebras $\widetilde{\pi}: \SS_H \to H_{\ab}$
such that $\pi = \widetilde{\pi} \circ t$.
The Hopf algebra $\SS_H$ becomes an $H_{\ab}$-comodule algebra 
with coaction 
\begin{equation}\label{eq-coaction}
\delta = (\id \otimes \, \widetilde{\pi}) \circ \Delta.
\end{equation}
On the generators of~$\SS_H$ the coaction is given by
\begin{equation*}
\delta(t_x) = \sum_{(x)} \, t_{x_{(1)}} \otimes \,\widetilde{\pi}(x_{(2)})
\quad\text{and}\quad
\delta(t^{-1}_x) = \sum_{(x)} \, t^{-1}_{x_{(2)}} \otimes \, \widetilde{\pi}\left( S(x_{(1)}) \right).
\end{equation*}

\begin{definition}
The \emph{algebra}~$\BB_H$ associated with a Hopf algebra~$H$ is the subalgebra of coinvariants\index{$A^{\co-H}$}
\index{coinvariant} of $\SS_H$ for this coaction:
\begin{equation*}
\BB_H = \SS_H^{\co-H_{\ab}} = \left\{ a\in \SS_H \, |\, \delta(a) = a \otimes 1\right\}.
\end{equation*}
\end{definition}

We call~$\BB_H$ the \emph{generic base algebra}\index{generic base algebra} of the Hopf algebra~$H$.
It has the following nice properties (see\,\cite[Th.\,3.6 and Cor.\,3.7]{KM1} and \cite[Prop.\,3.4]{KM2}).

\begin{theorem}\label{th-BH}
Let $H$ be a finite-dimensional Hopf algebra.

(a) The algebra~$\BB_H$ is a finitely generated smooth Noetherian domain; 
its Krull dimension\footnote{The Krull dimension of~$\BB_H$
is the dimension of the algebraic variety~$V$ such that $\BB_H = \OO(V)$.} is equal to~$\dim\ H$.

(b) $\SS_H$ is a finitely generated projective $\BB_H$-module.

(c) If in addition $H$ is pointed\index{Hopf algebra!pointed}, then 
\begin{equation*}
\BB_H = \CC[u_1^{\pm 1}, \ldots, u_{\ell}^{\pm 1}, u_{\ell+1}, \ldots, u_n],
\end{equation*}
where $n = \dim\ H$ and $\ell = \card \Gr(H)$ and  where $u_1, \ldots, u_n$ are monomials in the generators~$t_x$
of~$\Sym(t_H)$.
\end{theorem}

\begin{example}
If $H = \CC[G]$ be a group algebra\index{group algebra}\index{algebra!group}, then $H_{\ab} = \CC[\Gamma]$, 
where $\Gamma = G /[G,G]$ is the maximal abelian quotient of~$G$, i.e.
the quotient by the normal subgroup generated by all elements of the form $ghg^{-1}h^{-1}$.
Let $p: \ZZZ^{(G)} \to \Gamma$ be the homomorphism sending each generator~$t_g$ to the image
of~$g$ in~$\Gamma$. Let $Y_G$ be the kernel of~$p$. Then by\,\cite[Prop.\,9 and Prop.\,14]{AHN},
\begin{equation*}\index{$\BB_H$}
\BB_H = \CC[Y_G].
\end{equation*}
When $G$ is a finite group, then $Y_G$ is a free abelian subgroup of~$\ZZZ^{(G)}$ of finite index
(equal to the order of~$\Gamma$). A basis of~$Y_G$ is given in\,\cite[Lemma\,4.7]{KM2}
(see also\,\cite[App.\,A]{IK}).
\end{example}

\begin{example}\label{exa-BH}
For a Hopf algebra~$H$ it may happen that $H_{\ab} = \CC[\Gamma]$ is the algebra 
of an abelian group~$\Gamma$, for instance when the commutative Hopf algebra~$H_{\ab}$ 
is finite-dimensional and pointed (see\,\cite[Lemma\,2.1]{KM2}).
Then by Proposition\,\ref{prop-graded} the algebra~$\SS_H$ is $\Gamma$-graded with
$\SS_H = \bigoplus_{\gamma\in \Gamma} \, \SS_{H}(\gamma)$, where
\begin{equation*}
\SS_{H}(\gamma) = \left\{ a\in \SS_H \, |\, \delta(a) = a \otimes \gamma\right\},
\end{equation*}
and $\BB_H = \SS_{H}(0)$ is the component of~$\SS_H$ corresponding to the unit element~$0\in \Gamma$.
\end{example}

\begin{example}
Let $G$ be a finite group and $H = \OO(G)$. Since this Hopf algebra is commutative, we have $H_{\ab} = H$.
Therefore the morphism of Hopf algebras $\widetilde{\pi}: \SS_H \to H$ is split by the morphism
of coalgebras $t: H \to \SS_H$, i.e., $\widetilde{\pi} \circ t = \id_H$.
The coaction\,\eqref{eq-coaction} turns~$\SS_H$ into an $\OO(G)$-comodule algebra.
Thus by Proposition\,\ref{prop-G-alg}, $\SS_H$~is a $G$-algebra. One checks that
$G$ acts on $\SS_H = \CC[t_g]_{g\in G}[1/{\Theta_G}]$ by $g \cdot t_h = t_{gh}$ ($g,h \in G$)
and that the square~$\Theta_G^2$ of the Dedekind group determinant\index{group determinant} 
is $G$-invariant\index{invariant}. Therefore,
\begin{equation*}\index{$\BB_H$}
\BB_H = \CC[t_g]_{g\in G}^G \left[ \frac{1}{\Theta_G^2} \right],
\end{equation*}
where $\CC[t_g]_{g\in G}^G$ is the subalgebra of $G$-invariant polynomials.
\end{example}

The algebra~$\BB_H$ has also been completely described 
for the Sweedler algebra\index{Sweedler algebra} in\,\cite{AK} (see also\,\cite{Ka1}), 
for the Taft algebras\index{Taft algebra} and other natural generalizations of the Sweedler algebra in\,\cite{IK}.

\subsubsection{The algebra~$\AA_H$}\label{sssec-total-alg}

To construct what we call the 
\emph{generic $H$-Galois extension}\index{generic Hopf Galois extension}\index{Hopf Galois extension!generic}~$\AA_H$ 
we need the bilinear form $\sigma: H \times H \to \SS_H$ with values in~$\SS_H$\index{$\SS_H$} defined by
\begin{equation}\label{def-sigma}
\sigma(x,y) = \sum_{(x)(y)} \, t_{x_{(1)}} \, t_{y_{(1)}} \, t^{-1}_{x_{(2)}y_{(2)}} \, . \qquad (x,y \in H)
\end{equation}
By\,\cite[Prop.\,3.4]{KM2} the bilinear map~$\sigma$ actually takes values 
in the subalgebra~$\BB_H$ of~$\SS_H$.
We can then equip the vector space $\AA_H = \BB_H \otimes H$\index{$\AA_H$}\index{$\BB_H$} 
with the following product:\index{product}
\begin{equation}\label{def-*prod}
(b \otimes x) * (c \otimes y) =  \sum_{(x)(y)} \, bc\, \sigma(x_{(1)},y_{(1)})\, x_{(2)} y_{(2)}
\end{equation}
($b,c \in \BB_H$ and $x,y\in H$). 

The following properties of~$\AA_H$ were established in\,\cite{AK,KM1} (see also~\,\cite{Ka1}).

\begin{theorem}\label{th-AH}
Let $H$ be a finite-dimensional Hopf algebra. 

(a) The product $*$ turns $\AA_H$ into an associative unital algebra\index{algebra!associative}.

(b) The algebra~$\AA_H$ is a central $H$-Galois extension\index{Hopf Galois extension}\index{Hopf Galois extension!central} 
of~$\BB_H = \BB_H \otimes 1$
with coaction $\delta = \id_{\BB_H} \otimes \, \Delta$, where $\Delta$ is the coproduct of~$H$.
Moreover, $\AA_H$ is free as a $\BB_H$-module.

(c) Let $\chi_0: \BB_H \to \CC$ be the character defined as the restriction to~$\BB_H$ of the counit of~$\SS_H$.
Then there is an isomorphism of $H$-comodule algebras
\begin{equation*}
\CC \, \otimes_{\BB_H} \, \AA_H = \AA_H/ \ker(\chi_0)\, \AA_H \cong H.
\end{equation*}

(d) For any character\index{character of an algebra} $\chi : \BB_H \to \CC$ of~$\BB_H$, the fiber of~$\AA_H$ at~$\chi$
\begin{equation*}
\CC\,  \otimes_{\BB_H} \, \AA_H = \AA_H/ \ker(\chi)\, \AA_H 
\end{equation*}
is an $H$-Galois object\index{Galois object}.
\end{theorem}

This means that $\BB_H \subset \AA_H$ is a 
``non-commutative principal fiber bundle''\index{principal fiber bundle!non-commutative} with ``fiber''~$H$.
We can also see~$\AA_H$ as a deformation\index{deformation} of~$H$ over the parameter space~$\BB_H$
or, if one prefers, over the set $\Alg(\BB_H,\CC)$ of characters of~$\BB_H$.
By the last statement of the theorem, $\chi \mapsto \chi_*(\AA_H)$ induces a map $\Alg(\BB_H,\CC) \to \Gal_H(\CC)$.

\begin{exercise}
Check that the product\,\eqref{def-*prod} is associative with unit~$t_1^{-1} \otimes 1_H$.
\end{exercise}

\subsection{Multiparametric deformations of~$U_q\, \gs\gl(2)$ and of~$\gu_d$}\label{sec-deform-Uq}

We now illustrate the previous constructions on the cases where $H$ is the 
quantum enveloping algebra~$U_q\, = U_q\, \gs\gl(2)$ (defined in Sect.\,\ref{ssec-qea})
and its finite-dimensional quotients~$\gu_d$ (defined in Sect.\,\ref{ssec-ud}).
Both $U_q\,$ and~$\gu_d$ are pointed Hopf algebras.
Theo\-rems\,\ref{prop-def-Uq} and\,\ref{prop-def-ud} below are new.

\subsubsection{The generic base algebra of~$U_q\,$}\label{sssec-Uq-BH}

The Hopf algebra~$U_q\,$\index{$U_q\, \gs\gl(2)$} is infinite-dimen\-sion\-al with basis $\{E^i F^j K^{\ell}\}_{i,j \in \NN; \, \ell\in \ZZZ}$.
Its group~$\Gr(U_q\,)$\index{$\Gr(H)$} of group-like elements\index{group-like element} 
consists of all powers (positive and negative) of~$K$.
Therefore, by\,\eqref{SH-pointed} the free commutative Hopf algebra~$\SS_{U_q\,}$ is described by\index{$\SS_H$}
\begin{equation*}
\SS_{U_q\,} = \CC \left[ t_{E^i F^j K^{\ell}}\right]_{i,j \in \NN; \, \ell\in \ZZZ} \left[ \frac{1}{t_{K^m}}\right]_{m \in \ZZZ} .
\end{equation*}

The maximal commutative quotient Hopf algebra~$({U_q\,})_{\ab}$ is generated by four generators 
$\,\overline{\!E}$, $\,\overline{\!F}$, $\,\overline{\!K}$, $\,\overline{\!K}^{-1}$
subject to the same relations as the corresponding generators in~$U_q\,$ in Sect.\,\ref{ssec-qea}
plus the additional relations expressing that $({U_q\,})_{\ab}$ is commutative.
We thus have
\begin{equation*}
\overline{\!E} \,\overline{\!K} = \overline{\!K} \,\overline{\!E} = q^2 \,\overline{\!E}\,\overline{\!K} ,
\end{equation*}
which implies $\,\overline{\!E} = 0$ in~$({U_q\,})_{\ab}$ since $q^2 \neq 1$ and $\,\overline{\!K}$ is invertible.
Similarly, $\,\overline{\!F} = 0$.
Finally the relation
\begin{equation*}
\overline{\!K} - \overline{\!K}^{-1} 
= (q - q^{-1}) \, \left( \overline{\!E} \,\overline{\!F} - \overline{\!F}\,\overline{\!E} \right) = 0
\end{equation*}
shows that $\,\overline{\!K} = \overline{\!K}^{-1}$, hence $\,\overline{\!K}^2 = 1$ in~$({U_q\,})_{\ab}$. 
Therefore 
\[
({U_q\,})_{\ab} = \CC[\,\overline{\!K}]/(\,\overline{\!K}^2-1) \cong \CC[\ZZZ/2],
\]
\index{$H_{\ab}$}
which is the algebra of the group~$\ZZZ/2$. 

As noted in Example\,\ref{exa-BH}, the isomorphism $({U_q\,})_{\ab}\cong \CC[\ZZZ/2]$ implies that
$\SS_{U_q\,}$ is a superalgebra\index{superalgebra}: $\SS_{U_q\,} = \SS_{U_q\,}(0) \bigoplus \SS_{U_q\,}(1)$, 
and that the generic base algebra\index{generic base algebra}~$\BB_{U_q\,}$\index{$\BB_H$}
coincides with the $0$-degree component:
\[
\BB_{U_q\,} = \SS_{U_q\,}(0).
\]

On the generators $t_E$, $t_F$, $t_K$ the coproduct of~$\SS_{U_q\,}$ is given by 
\begin{equation*}
\Delta(t_E) = t_1 \otimes t_E + t_E \otimes t_K, 
\;\; \Delta(t_F) = t_{K^{-1}} \otimes t_F + t_F \otimes t_1,
\;\; \Delta(t_K) = t_K \otimes t_K .
\end{equation*}
Since $\widetilde{\pi}(t_E) = \overline{\!E} = 0$, $\widetilde{\pi}(t_F) = \overline{\!F} = 0$, 
and $\widetilde{\pi}(t_K) = \overline{\!K}$,
the coaction~$\delta$ of~$({U_q\,})_{\ab}$ on~$\SS_{U_q\,}$ satisfies
\begin{equation*}
\delta(t_E) =  t_E \otimes \,\overline{\!K}, \quad 
\Delta(t_F) = t_F \otimes 1, \quad 
\Delta(t_K) = t_K \otimes \,\overline{\!K} .
\end{equation*}
Therefore, $t_F$ is an even element, i.e. it belongs to~$\SS_{U_q\,}(0) = \BB_{U_q\,}$ while $t_E$ and~$t_K$ are both odd, 
that is belong to~$\SS_{U_q\,}(1)$.
It can be proved more generally that $t_{E^i F^j K^{\ell}}$ belongs to~$\BB_{U_q\,}$ if and only if $i+ \ell$ is even, 
and that $t^{-1}_{K^m}$ belongs to~$\BB_{U_q\,}$ if and only $m$ is even. 

\begin{exercise}\label{exo-BUq}
Set $u_{E^iF^jK^{\ell}} = t_{E^iF^jK^{\ell}}$ if $i+ \ell$ is even, and 
$u_{E^iF^jK^{\ell}} = t_{E^iF^jK^{\ell}} \, t^{-1}_K$ if $i+ \ell$ is odd.
Show that 
\begin{equation*}
\BB_{U_q\,} = \CC \left[ u_{E^i F^j K^{\ell}}\right]_{i,j \in \NN; \, \ell\in \ZZZ} \left[ \frac{1}{u_{K^m}}\right]_{m \in \ZZZ} .
\end{equation*}
\end{exercise}

\subsubsection{The algebra~$\AA_{U_q\,}$}\label{sssec-Uq-AH}

We have the following result.\index{$\AA_H$}

\begin{theorem}\label{prop-def-Uq}
The generic $U_q\,$-Galois extension\index{generic Hopf Galois extension}
\index{Hopf Galois extension!generic}~$\AA_{U_q\,}$ 
is the $\BB_H$-algebra generated by $E$, $F$, $K$, $K^{-1}$ subject to the relations
\begin{equation*}
K * K^{-1} = K^{-1} * K = \frac{t_K t_{K^{-1}}}{t_1},
\end{equation*}
\begin{equation*}
K*E = q^2 \, E* K + (1-q^2) \frac{t_E}{t_K} \, K * K,
\end{equation*}
\begin{equation*}
K*F = q^{-2} \, F*K + (1-q^{-2}) \, t_F \, K,
\end{equation*}
\begin{equation*}
E*F - F*E = t_1 \frac{(t_{K^{-1}}/t_K)\, K - K^{-1}}{q - q^{-1}} 
+ (q^{-2}-1) \left( \frac{t_E}{t_K} \, F* K - \frac{t_Et_F}{t_K} \, K \right) .
\end{equation*}
\end{theorem}

The algebra~$\AA_{U_q\,}$ is an $U_q\,$-comodule algebra with coaction given by the same formulas 
as for the coproduct of~$U_q\,$. 
The algebra depends continuously on the parameters $t_E$,~$t_F$ which can take any complex values
and on the parameters $t_1$, $t_K$, $t_{K^{-1}}$ which can take any \emph{non-zero} complex values. 
Note that all monomials in the $t$-variables occurring in the previous relations belong to~$\BB_{U_q\,}$
(they are all of degree~$0$ in the superalgebra~$\SS_{U_q\,}$).

If we specialize the parameters $t_1$, $t_K$, $t_{K^{-1}}$ to~$1$ and the parameters $t_E$, $t_F$ to~$0$,
we recover the defining relations of~$U_q\,$ and $\AA_{U_q\,}$ becomes~$U_q\,$.
In other words, $\AA_{U_q\,}$ is a \emph{$5$-parameter deformation}\index{deformation} 
of~$U_q\,$ as a non-commutative principal bundle\index{principal fiber bundle!non-commutative}.

\begin{proof}
We use an observation made in\,\cite[Sect.\,6]{AK}: in order to find relations between elements $1 \otimes x$ in~$\AA_H$,
where $x$ is an arbitrary element of a Hopf algebra~$H$, it is enough to find the relations between the following
elements of the tensor product algebra~$\BB_H \otimes H$:
\begin{equation*}
X_x = \sum_{(x)} \, t_{x_{(1)}} \otimes x_{(2)} .
\end{equation*}
It follows from the formula for the coproduct of~$U_q\,$ (see Sect.\,\ref{ssec-qea}) that we have
\begin{equation*}
X_1 = t_1 \, 1, \qquad
X_K = t_K \, K, \qquad
X_{K^{-1}} = t_{K^{-1}} \, K^{-1}, 
\end{equation*}
\begin{equation*}
X_E = t_1 \, E + t_E \, K, \qquad
X_F = t_{K^{-1}} \, F + t_F\, 1.
\end{equation*}
(Here we dropped the tensor product signs since we may consider the commutative algebra~$\BB_H$
as an extended algebra of scalars.)

To prove the relations between $K$ and $K^{-1}$, it suffices to compute $X_KX_{K^{-1}}$ and $X_{K^{-1}}X_K$. We have
\begin{equation*}
X_KX_{K^{-1}} = t_K t_{K^{-1}} \, K K^{-1} = t_K t_{K^{-1}} = \frac{t_K t_{K^{-1}}}{t_1}\, X_1,
\end{equation*}
which is also equal to~$X_{K^{-1}} X_K$; this implies the desired formulas for~$K * K^{-1}$ and~$K^{-1} * K$. 

For the relation between $K$ and~$E$ in~$\AA_H$, it is enough to compute the following:
\begin{eqnarray*}
X_K X_E - q^2\, X_E X_K
& = & t_Kt_1 \, K E + t_Kt_E \, K^2 - q^2\, t_1 t_K EK - q^2\, t_Et_K \, K^2 \\
& = & t_1t_K \, (K E - q^2\, EK) + (1-q^2)\, t_Et_K\, K^2 \\
& = & (1-q^2)\, t_Et_K\, K^2.
\end{eqnarray*}
Now, $(X_K)^2 = t_K^2 \, K^2$.
Therefore, 
\begin{equation*}
X_K X_E - q^2\, X_E X_K = (1-q^2)\, t_Et_K/t_K^2 \, (X_K)^2 = (1-q^2)\, t_E/t_K \, (X_K)^2 .
\end{equation*}

We leave the computation of the relation between $K$ and~$F$ in~$\AA_H$ as an exercise to the reader.
For the commutator of $E$ and $F$ in~$\AA_H$, we have
\begin{eqnarray*}
X_E X_F - X_F X_E 
& = & (t_1 \, E + t_E \, K) (t_{K^{-1}} \, F + t_F\, 1) - (t_{K^{-1}} \, F + t_F\, 1) (t_1 \, E + t_E \, K)  \\
& = & t_1 t_{K^{-1}} \, (EF-FE) + (q^{-2} - 1) \, t_Et_{K^{-1}} FK \\
& = & \frac{1}{q - q^{-1}} \, t_1 t_{K^{-1}} \, (K - K^{-1}) + (q^{-2} - 1) \, t_Et_{K^{-1}} FK \\
& = & \frac{1}{q - q^{-1}} \, t_1 \left( \frac{t_{K^{-1}}}{t_K} \, X_K - X_{K^{-1}} \right) + (q^{-2} - 1) \, t_Et_{K^{-1}} FK .
\end{eqnarray*}
It remains to compute $FK$ in terms of the $X$-variables. We have
\begin{equation*}
X_F X_K = t_K t_{K^{-1}} \, FK + t_Ft_K\, K = t_K t_{K^{-1}} \, FK + t_F\, X_K,
\end{equation*}
so that
\begin{equation*}
t_Et_{K^{-1}} FK = \frac{t_E}{t_K}\, X_F X_K -  \frac{t_Et_F}{t_K} \, X_K.
\end{equation*}
Combining these equalities, we obtain a formula for $X_E X_F - X_F X_E$ in terms of the $X$-variables,
hence the desired formula for~$E * F - F * E$.
\qed
\end{proof}

\subsubsection{A deformation of~$\gu_d$}\label{sssec-ud}\index{deformation}

Let $q$ be a root of unity\index{root of unity} of order $d\geq 3$.
Consider the finite-dimensional Hopf algebra~$\gu_d$\index{$\gu_d$}
defined in Sect.\,\ref{ssec-ud}. We know that it has a basis consisting of the $e^3$~elements
$E^i F^j K^{\ell}$, where $1 \leq i,j,\ell \leq e-1$. 
Recall that $e= d/2$ if $d$ is even and $e=d$ if $d$ is odd. 
The group~$\Gr(\gu_d)$\index{$\Gr(H)$}\index{group-like element}
consists of the $e$~elements $1, K, K^2, \ldots, K^{e-1}$; it is a cyclic group of order~$e$.

By\,\eqref{SH-pointed} the free commutative Hopf algebra~$\SS_{\gu_d}$\index{$\SS_H$} is given by
\begin{equation*}
\SS_{\gu_d} = \CC \left[ t_{E^i F^j K^{\ell}} \right]_{0 \leq i,j , \ell \leq e-1} 
\left[ \frac{1}{t_{K^m}}\right]_{0 \leq m \leq e-1} .
\end{equation*}

The maximal commutative quotient Hopf algebra~$({\gu_d})_{\ab}$\index{$H_{\ab}$} is the quotient of $({U_q\,})_{\ab}$
by the additional relation $\,\overline{\!K}^e = 1$.
Since $\,\overline{\!K}^2 = 1$, we conclude that
\begin{equation*}
({\gu_d})_{\ab} =
\begin{cases}
\hskip 33pt \CC &  \text{if}\; e \; \text{is odd}, \\
\noalign{\smallskip}
({U_q\,})_{\ab} \cong \CC[\ZZZ/2] & \text{if}\; e \; \text{is even}.
\end{cases}
\end{equation*}
Therefore, if $e$ is odd, then $\SS_{\gu_d}$ is trivially graded, which implies $\BB_{\gu_d} = \SS_{\gu_d}$.
If~$e$ is even, then $\SS_{\gu_d}$ is a superalgebra\index{superalgebra} and 
the generic base algebra\index{generic base algebra} is $\BB_{\gu_d}$\index{$\BB_H$} is its even part
(see Exercise\,\ref{exo-Bud} below for a complete description).

\begin{theorem}\label{prop-def-ud}
The algebra~$\AA_{\gu_d}$ is the quotient of~$\AA_{U_q\,}$ by the two-sided ideal generated by the relations
\begin{equation*}
K^{*e} - \frac{t_K^e}{t_1} = 0, \qquad
\left(E - \frac{t_E}{t_K} K \right)^{*e} \!\! = 0, \qquad
\left(F - \frac{t_F}{t_1}  \right)^{*e} \!\! = 0 .
\end{equation*}
\end{theorem}

If we set $t_1 = t_K = t_{K^{-1}} = 1$ and $t_E = t_F = 0$ in the defining relations of~$\AA_{\gu_d}$
(see Theorems\,\ref{prop-def-Uq} and\,\ref{prop-def-ud}), we recover those of~${\gu_d}$.

\begin{proof}
We proceed as in the proof of Theorem\,\ref{prop-def-Uq} by checking the relations between the corresponding 
$X$-variables in~$\BB_{\gu_d} \otimes {\gu_d}$.
We have
\begin{equation*}
(X_K)^e - \frac{t_K^e}{t_1}X_1 = t_K^e \, K^e - t_K^e = 0
\end{equation*}
since $K^e = 1$ in~$\gu_d$. 
Next, in view of $E^e = F^e = 0$ in~$\gu_d$, we have
\begin{equation*}
\left(X_E - \frac{t_E}{t_K} \, X_K \right)^{*e} = t_1^e \, E^e = 0
\quad
\text{and} 
\quad
\left(X_F - \frac{t_F}{t_1} \, X_1 \right)^{*e} = t_{K^{-1}}^e \, F^e = 0 .
\end{equation*}
This completes the proof.
\qed
\end{proof}

Let us determine the ``parameter space''~$\Alg(\BB_{\gu_d},\CC)$ when $e$ is odd. In this case,
$\BB_{\gu_d} = \SS_{\gu_d}$. 
Since $\SS_{\gu_d} = \CC \left[ t_{E^i F^j K^{\ell}} \right]_{0 \leq i,j , \ell \leq e-1} 
[ {1}/{t_{K^m}}]_{0 \leq m \leq e-1}$,
a character of~$\BB_{\gu_d}$ is completely determined by its values on the generators $t_{E^i F^j K^{\ell}}$; 
each of these generators can take any complex value, 
except in the case $(i,j) = (0,0)$, where the corresponding value has to be non-zero.
It follows that 
\begin{equation*}
\Alg(\BB_{\gu_d},\CC) \cong \CC^{e(e^2-1)} \times (\CC^{\times})^{e} ,
\end{equation*}
which is an open Zarisky subset of the affine space of dimension~$e^3$.

\begin{exercise}\label{exo-Bud}
Assume $e$ is even (equivalently, $d$ is divisible by~$4$). 
Define $u_{E^iF^jK^{\ell}}$ as in Exercise\,\ref{exo-BUq}.
Show that 
\begin{equation*}
\BB_ {\gu_d} = \CC \left[ u_{E^i F^j K^{\ell}} \right]_{0 \leq i,j , \ell \leq e-1} [ {1}/{u_{K^m}}]_{0 \leq m \leq e-1}.
\end{equation*}
Hence, $\Alg(\BB_{\gu_d},\CC) \cong \CC^{e(e^2-1)} \times (\CC^{\times})^{e}$ holds in this case too.
\end{exercise}

\subsection*{Acknowledgements}
I thank the organizers of the Summer school 
``Geo\-metric, topological and algebraic methods for quantum field theory'' held at Villa de Leyva, Colombia in July 2015
for inviting me to give the course that led to these notes. I am also grateful to the students for their feedback
and to Sylvie Paycha for her careful reading of these notes and her suggestions.
Finally let me mention the travel support I received from the Institut de Recherche Math\'e\-ma\-tique Avanc\'ee
in Strasbourg.


\printindex

\end{document}